\documentclass[10pt,british,english,reqno]{amsart}
\usepackage[T1]{fontenc}
\usepackage[latin1]{inputenc}
\usepackage{xcolor}
\usepackage{babel}
\usepackage{amsthm}
\usepackage{amstext}
\usepackage{amssymb}
\usepackage{esint}
\usepackage[all]{xy}
\PassOptionsToPackage{normalem}{ulem}
\usepackage{ulem}
\usepackage[unicode=true,pdfusetitle,
 bookmarks=true,bookmarksnumbered=false,bookmarksopen=false,
 breaklinks=false,pdfborder={0 0 0},backref=false,colorlinks=true]
 {hyperref}
\usepackage{breakurl}

\makeatletter

\providecolor{lyxadded}{rgb}{0,0,1}
\providecolor{lyxdeleted}{rgb}{1,0,0}

\numberwithin{equation}{section}
\numberwithin{figure}{section}
\usepackage{enumitem}		
\newlength{\lyxlistindent}      
\theoremstyle{plain}
\newtheorem{thm}{\protect\theoremname}
  \theoremstyle{definition}
  \newtheorem{defn}[thm]{\protect\definitionname}
  \theoremstyle{definition}
  \newtheorem{example}[thm]{\protect\examplename}
  \theoremstyle{plain}
  \newtheorem{cor}[thm]{\protect\corollaryname}
  \theoremstyle{remark}
  \newtheorem*{rem*}{\protect\remarkname}
  \theoremstyle{plain}
  \newtheorem{lem}[thm]{\protect\lemmaname}
  \theoremstyle{plain}
  \newtheorem*{lem*}{\protect\lemmaname}

\@ifundefined{date}{}{\date{}}
\usepackage{mathabx}
\usepackage{pdfsync}
\usepackage{graphics}
\usepackage{epstopdf}
\usepackage[all]{xy}
\usepackage{amscd}
\usepackage{mathabx}

\newcommand{\xyR}[1]{ \makeatletter
\xydef@\xymatrixrowsep@{#1} \makeatother} 
\newcommand{\xyC}[1]{ \makeatletter
\xydef@\xymatrixcolsep@{#1} \makeatother} 
\entrymodifiers={++[ ][F]} 

\newcommand{\ra}{\longrightarrow}

\newcommand{\field}[1]{\mathbb{#1}}
\newcommand{\R}{\field{R}} 
\newcommand{\N}{\field{N}} 
\newcommand{\hyperR}{{^*\R}} 
\newcommand{\Z}{\ensuremath{\mathbb{Z}}} 

\newcommand{\mZ}{\mathcal{Z}}
\newcommand{\D}{\mathcal{D}}

\newcommand{\eps}{\varepsilon} 
\renewcommand{\phi}{\varphi} 
\newcommand{\diff}[1]{\,\hbox{\rm d}#1} 

\newcommand{\Coo}{\mbox{\ensuremath{\mathcal{C}}}^{\infty}} 

\newcommand{\Rtil}{\widetilde \R} 
\newcommand{\otilc}{\widetilde \Omega_c} 
\newcommand{\gss}{{\mathcal G}^{\text{\rm s}}} 
\newcommand{\gse}{{\mathcal G}^{\text{\rm e}}} 
\newcommand{\gsd}{{\mathcal G}^{\text{\rm d}}} 
\newcommand{\gsh}{\hat{\mathcal{G}}} 
\newcommand{\erm}{\text{\rm e}} 
\newcommand{\srm}{\text{\rm s}}
\newcommand{\drm}{\text{\rm d}}

\newcommand{\ptind}{\displaystyle \mathop {\ldots\ldots\,}} 
\newcommand{\then}{\quad \Longrightarrow \quad} 

\newcommand{\maC}{\mathcal{C}}

\newcommand{\ep}{\varepsilon}

\newcommand{\B}{\mathcal{B}}
\newcommand{\GBZ}{\mathcal{G}(\mathcal{B},\mathcal{Z},\Omega)}

\newcommand{\ue}{(u_{\eps})}

\newcommand{\re}{\rho_{\eps}}

\makeatother

  \addto\captionsbritish{\renewcommand{\corollaryname}{\inputencoding{latin9}Corollary}}
  \addto\captionsbritish{\renewcommand{\definitionname}{\inputencoding{latin9}Definition}}
  \addto\captionsbritish{\renewcommand{\examplename}{\inputencoding{latin9}Example}}
  \addto\captionsbritish{\renewcommand{\lemmaname}{\inputencoding{latin9}Lemma}}
  \addto\captionsbritish{\renewcommand{\remarkname}{\inputencoding{latin9}Remark}}
  \addto\captionsbritish{\renewcommand{\theoremname}{\inputencoding{latin9}Theorem}}
  \addto\captionsenglish{\renewcommand{\corollaryname}{\inputencoding{latin9}Corollary}}
  \addto\captionsenglish{\renewcommand{\definitionname}{\inputencoding{latin9}Definition}}
  \addto\captionsenglish{\renewcommand{\examplename}{\inputencoding{latin9}Example}}
  \addto\captionsenglish{\renewcommand{\lemmaname}{\inputencoding{latin9}Lemma}}
  \addto\captionsenglish{\renewcommand{\remarkname}{\inputencoding{latin9}Remark}}
  \addto\captionsenglish{\renewcommand{\theoremname}{\inputencoding{latin9}Theorem}}
  \providecommand{\corollaryname}{\inputencoding{latin9}Corollary}
  \providecommand{\definitionname}{\inputencoding{latin9}Definition}
  \providecommand{\examplename}{\inputencoding{latin9}Example}
  \providecommand{\lemmaname}{\inputencoding{latin9}Lemma}
  \providecommand{\remarkname}{\inputencoding{latin9}Remark}
\providecommand{\theoremname}{\inputencoding{latin9}Theorem}

\begin{document}

\title{Asymptotic gauges: Generalization of Colombeau type algebras}

\author{Paolo Giordano \and Lorenzo Luperi Baglini}

\thanks{P.\ Giordano has been $\text{supp}$orted by grant P25116-N25 of
the Austrian Science Fund FWF}

\address{\textsc{Faculty of Mathematics, University of Vienna, Austria, Oskar-Morgenstern-Platz
1, 1090 Wien, Austria}}

\thanks{L.\ Luperi Baglini has been $\text{supp}$orted by grant P25311-N25
of the Austrian Science Fund FWF}

\email{\texttt{paolo.giordano@univie.ac.at, }\foreignlanguage{english}{\texttt{lorenzo.luperi.baglini@univie.ac.at}}}


\begin{abstract}
We use the general notion of set of indices to construct algebras
of nonlinear generalized functions of Colombeau type. They are formally
defined in the same way as the special Colombeau algebra, but based
on more general ``growth condition'' formalized by the notion of
asymptotic gauge. This generalization includes the special, full and
nonstandard analysis based Colombeau type algebras in a unique framework.
We compare Colombeau algebras generated by asymptotic gauges with
other analogous construction, and we study systematically their properties,
with particular attention to the existence and definition of embeddings
of distributions. We finally prove that, in our framework, for every
linear homogeneous ODE with generalized coefficients there exists
a minimal Colombeau algebra generated by asymptotic gauges in which
the ODE can be uniquely solved. This marks a main difference with
the Colombeau special algebra, where only linear homogeneous ODEs
satisfying some restriction on the coefficients can be solved.
\end{abstract}
\maketitle

\section{Introduction}

Currently, a successful approach to modeling singularities as generalized
functions in a nonlinear context is the theory of Colombeau-type algebras
of generalized functions. The basic underlying idea is to regularize
distributions (or even more singular quantities) through nets of smooth
functions depending on a regularization parameter $\eps$ and then
to quantify asymptotically the strength of singularities in terms
of this parameter $\eps$. \foreignlanguage{british}{In particular,
these algebras contain the space of Schwartz distributions as a linear
subspace and the algebra of smooth functions as a faithful subalgebra.
We suppose the reader to have a certain familiarity with this topic
and refer to \cite{Col84,Col85,Biag,Col92,MO92,GKOS} for detailed
information; as for terminology and notations we mainly follow \cite{GKOS}.}

Since the beginning of this theory, it was natural to generalize Colombeau
construction replacing the family $\left(\eps^{n}\right)_{\eps\in(0,1],n\in\N}$
with different ``scales''. So, we have the notions of asymptotic
scales (see e.g.\ \cite{DelSca98,DelSca00}), $(\maC,\mathcal{E},\mathcal{P})$-algebras
(see e.g.\ \cite{Del09,Has11} and references therein), and sequence
spaces with exponent weights (\cite{DPHV}).

In realizing this generalization, one can ask problems like:
\begin{itemize}
\item When do we obtain an algebra? 
\item When can we embed Schwartz distributions using the common method of
regularization by means of a mollifier?
\item How can we use this generalization to solve differential problems
having a singular growth, i.e.\ growing more than polynomially with
$\eps$?
\end{itemize}
The present work inscribes in this research thread. It is hence natural
to clarify the relationship between our approach and the cited articles,
and to highlight what we obtain more with respect to them. In this
introduction, we start this clarification, so that the reader can
more easily understand the general picture.

\medskip{}

\textbf{Smooth functions:} First of all, a general approach is frequently
preferred, e.g.\ by considering a sheaf $\mathcal{E}$ of topological
algebras, and suitable families of seminorms. On the contrary, in
the present work we will only consider the usual sheaf of smooth functions
$\Coo(\Omega)$, for $\Omega$ open in $\R^{n}$, and the usual family
of seminorms $\sup_{x\in K}\left|\partial^{\alpha}f(x)\right|$, where
$K\Subset\Omega$ and $\alpha\in\N^{n}$. In spite of our choice,
as it is clearly stated in \cite[pag. 394]{Del09}: ``Except in a
few cases, the sheaf $\mathcal{E}$ is chosen to be the sheaf of smooth
functions''. In this way, we can focus on the conditions we need
to impose to the more general family of scales. We can thus avoid
to add further conditions which trivially hold in the (almost) unique
case that most of the readers will consider. We hence left to the
interested reader the natural generalization of the present work to
a more abstract framework.\medskip{}

\textbf{Set of indices:} The results of this paper are proved for
a generic set of indices, a new unifying structure introduced in \cite{GiNi14}.
This permits to include several Colombeau algebras in the same framework
and notations: the special algebra $\gss$, the full algebra $\gse$
and the nonstandard analysis based algebra of asymptotic functions
$\gsh$. Even if in the present article we are going to develop only
the cases $\gss$, $\gse$ and $\gsh$, we are strongly convinced
that with minor modifications (see \cite{GiNi14}), the results we
are going to present can also be applied to the diffeomorphism invariant
algebras $\gsd$, $\mathcal{G}^{2}$ and $\hat{\mathcal{G}}$ of \cite{GKOS}.

Both asymptotic scales and $(\maC,\mathcal{E},\mathcal{P})$-algebras
apply only to the special case. \foreignlanguage{british}{Sequence
spaces with exponent weights} applies to $\gss$, $\gse$, $\gsd$
but not to $\gsh$.\medskip{}

\textbf{Logical structure:} One of the key features in using a set
of indices is that for all the algebras $\gss$, $\gse$, $\gsh$,
$\gsd$, $\mathcal{G}^{2}$ and $\hat{\mathcal{G}}$ of \cite{GKOS},
we have the same logical structure of the simple Colombeau algebra,
i.e.\ $\forall K\,\forall\alpha\,\exists N$ and $\forall K\,\forall\alpha\,\forall m$,
followed by a suitable big-O asymptotic relation. Both $\gse$ and
$\gsd$ can be seen as \foreignlanguage{british}{sequence spaces with
exponent weights}, but continuing to use the usual more involved logical
structure (hidden by the use of infinite intersections and unions).
On the other hand, this permits to \cite{DPHV} to underscore the
relationship between sequence spaces with exponent weights and Maddox
sequence spaces.

\medskip{}

\textbf{Scales as primitive data:} Like in \cite{DelSca98,DPHV},
we take the choice of the ``scale'' as one of the primitive data.
This approach is methodologically a little different from $(\maC,\mathcal{E},\mathcal{P})$-algebras,
where the scale is hidden in the pair $(A,I)$ of the ring $A$ of
moderate and the ideal $I$ of negligible scalars, but where the ring
$A$ usually contains much more than only the scales. For example,
when $A$ is polynomially overgenerated, it contains both infinitesimals
and infinite nets. Moreover, polynomially overgenerated rings do not
permit to obtain an algebra closed with respect to exponential (see
\cite{Del09}). This represents a limitation e.g.\ in solving even
linear ODE with generalized constant coefficients. On the contrary,
we prove in Thm. \ref{thm:linearConstantCoeffODE} that every linear
ODE with constant (generalized) coefficients whose scale is of type
$\mathcal{B}$ has a unique solution whose scale is of type $e^{\mathcal{B}}$
(see Def. \ref{def:expB}). Moreover, the Colombeau-like algebra defined
starting from the scale $e^{\mathcal{B}}$ is the smallest among this
type of algebras where every ODE of this type has a solution (see
Thm. \ref{thm:expB-smallest}). \medskip{}

\textbf{Embedding of distributions:} In Thm. \ref{thm:principalNecessary},
we characterize which scales permit to embed distributions using a
mollifier and respecting the product of smooth functions. Our results
also clarify when, in other approaches, embeddings of distributions
are possible and when they are not. About this problem, see also \cite[pag. 1-2]{DelSca00}
and \cite{DPHV}.\medskip{}

\textbf{Generality:} Considering \cite{Has11,DPHV}, it is already
known that the $(\maC,\mathcal{E},\mathcal{P})$-algebras approach
is the most general one. In section \ref{sub:A-comparison}, we prove
that when we consider the usual sheaf of smooth functions, and we
consider only the special algebra case, the $(\maC,\mathcal{E},\mathcal{P})$-algebras
approach is equivalent to our approach.

\subsection{Set of indices}

In this section, we recall notations and notions from \cite{GiNi14}
that we will use in the present work. For all the proofs, we refer
to \cite{GiNi14}. In the naturals $\N$ we always include zero.

If $\phi\in\D(\R^{n})$, $r\in\R_{>0}$ and $x\in\R^{n}$, we use
the notations $r\odot\phi$ for the function $x\in\R^{n}\mapsto\frac{1}{r^{n}}\cdot\phi\left(\frac{x}{r}\right)\in\R$
and $x\oplus\phi$ for the function $y\in\R^{n}\mapsto\phi(y-x)\in\R$.
These new notations permit to highlight that $\odot$ is a free action
of the multiplicative group $(\R_{>0},\cdot,1)$ on $\D(\R^{n})$
and $\oplus$ is a free action of the additive group $(\R_{>0},+,0)$
on $\D(\R^{n})$. We also have the distributive property $r\odot(x\oplus\phi)=rx\oplus r\odot\phi$.
\begin{defn}
\label{def:setOfIndices}We say that $\mathbb{I}=(I,\le,\mathcal{I})$
is a \emph{set of indices} if the following conditions hold:
\begin{enumerate}[leftmargin=*,label=(\roman*),align=left ]
\item \foreignlanguage{british}{\label{enu:DefSoI-preorder}$(I,\le)$
is a pre-ordered set, i.e.\ it is a non empty set $I$ with a reflexive
and transitive relation $\le$.}
\item \label{enu:DefSoI-union}$\mathcal{I}$ is a set of subsets of $I$
such that $\emptyset\notin\mathcal{I}$ and $I\in\mathcal{I}$.
\item \label{enu:DefSoI-intersection}$\forall A,B\in\mathcal{I}\,\exists C\in\mathcal{I}:\ C\subseteq A\cap B$.
\end{enumerate}

\noindent For all $e\in I$, set $(\emptyset,e]:=\left\{ \eps\in I\mid\eps\le e\right\} $.
As usual, we say $\eps<e$ if $\eps\le e$ and $\eps\ne e$. Using
these notations, we state the last condition in the definition of
set of indices:
\begin{enumerate}[%
leftmargin={*},label=(\roman*),align=left,start=4%
]
\item \label{enu:DefSoI-DonwDir}If $e\le a\in A\in\mathcal{I}$, the set
$A_{\le e}:=(\emptyset,e]\cap A$ is downward directed by $<$, i.e.,
it is non empty and 
\begin{equation}
\forall b,c\in A_{\le e}\,\exists d\in A_{\le e}:\ d<b\ ,\ d<c.\label{eq:strictlyDownDirected}
\end{equation}

\end{enumerate}
\end{defn}
Henceforward, functions of the type $f:I\ra\R$ will also be called
\emph{nets}, and for their evaluation we will both use the notations
$f_{\eps}$ or $f(\eps)$, in case the subscript notation is too cumbersome.
When the domain $I$ is clear, we use also the notation $f=\left(f_{\eps}\right)$
for the whole net. Analogous notations will be used for nets of smooth
functions $u=(u_{\eps})\in\mathcal{C}(\Omega)^{I}$.
\begin{example}
\label{exa:stdAndNS}\ 
\begin{enumerate}[leftmargin=*,label=(\roman*),align=left ]
\item Conditions \ref{enu:DefSoI-union} and \ref{enu:DefSoI-intersection}
can be summarized saying that $\mathcal{I}$ is a filter base on $I$
which contains $I$.
\item \label{enu:I^s}The simplest example of set of indices is given by
$I^{\srm}:=(0,1]\subseteq\R$, the relation $\le$ is the usual order
relation on $\R$, and $\mathcal{I}^{\srm}:=\left\{ (0,\eps_{0}]\mid\eps_{0}\in I\right\} $.
We denote by $\mathbb{I}^{\srm}:=(I^{\srm},\le,\mathcal{I}^{\srm})$
this set of indices which, of course, is that used for the special
algebra $\gss$.
\item \label{enu:Ihat}In the context of \cite{ToVe08}, we set $\hat{I}:=\D_{0}=\D(\R^{d})$.
The pre-order relation is defined by $\phi\le\psi$ iff $\underline{\phi}\le\underline{\psi}$,
where $\underline{\phi}:=\text{diam}\left(\text{supp}(\phi)\right)$
if $\phi\ne0$ and $\underline{\phi}:=1$ otherwise. $\hat{\mathcal{I}}$
is the free ultrafilter on $\D_{0}$ used in \cite{ToVe08}. Then
$\hat{\mathbb{I}}:=(\hat{I},\le,\hat{\mathcal{I}})$ is a set of indices.
\item \label{enu:setOfIndicesFullAlgebra}With the usual notations of \cite{GKOS}
for the full algebra $\gse$, we define $I^{\erm}:=\mathcal{A}_{0}$,
$\mathcal{I}^{\erm}:=\left\{ \mathcal{A}_{q}\mid q\in\N\right\} $,
and for $\eps$, $e\in I^{\erm}$, we define $\eps\le e$ iff there
exists $r\in\R_{>0}$ such that $r\le1$ and $\eps=r\odot e$. Then
$\mathbb{I}^{\erm}:=(I^{\erm},\le,\mathcal{I}^{\erm})$ is a set of
indices.
\end{enumerate}
\end{example}
As we mentioned in the introduction, in the present work we will actually
consider only these examples of set of indices.

In each set of indices, we can define two notions of big-O that formally
behave in the usual way. Since each set of the form $A_{\le e}=(\emptyset,e]\cap A$
is downward directed, the first big-O is the usual one:
\begin{defn}
\label{def:usualBigOh}Let $\mathbb{I}=(I,\le,\mathcal{I})$ be a
set of indices. Let $a\in A\in\mathcal{I}$ and $(x_{\eps})$, $(y_{\eps})\in\R^{I}$
be two nets of real numbers defined in $I$. We write
\begin{equation}
x_{\eps}=O_{a,A}(y_{\eps})\ \text{as }\eps\in\mathbb{I}\label{eq:usualBigOh}
\end{equation}
if
\begin{equation}
\exists H\in\R_{>0}\,\exists\eps_{0}\in A_{\le a}\,\forall\eps\in A_{\le\eps_{0}}:\ |x_{\eps}|\le H\cdot|y_{\eps}|.\label{eq:1stBigOhDef}
\end{equation}

\end{defn}
\ 
\begin{defn}
\label{def:2ndBigOh}Let $\mathbb{I}=(I,\le\mathcal{I})$ be a set
of indices. Let $\mathcal{J}\subseteq\mathcal{I}$ be a non empty
subset of $\mathcal{I}$ such that
\begin{equation}
\forall A,B\in\mathcal{J}\,\exists C\in\mathcal{J}:\ C\subseteq A\cap B.\label{eq:hypJ}
\end{equation}
Finally, let $(x_{\eps})$, $(y_{\eps})\in\R^{I}$ be nets of real
numbers. Then we say
\[
x_{\eps}=O_{\mathcal{J}}(y_{\eps})\text{ as }\eps\in\mathbb{I}
\]
if
\[
\exists A\in\mathcal{J}\,\forall a\in A:\ x_{\eps}=O_{a,A}(y_{\eps}).
\]
We simply write $x_{\eps}=O(y_{\eps})$ (as $\eps\in\mathbb{I}$)
when $\mathcal{J=I}$, i.e.\ to denote $x_{\eps}=O_{\mathcal{I}}(y_{\eps})$.
\end{defn}
The simplification consequent to the use of the second notion of big-O
is due to the following theorem, which states that also the second
big-O formally behaves as expected:
\begin{thm}
\label{thm:2ndBigOh-prop}Under the assumptions of Def. \ref{def:2ndBigOh},
the following properties of $O_{\mathcal{J}}$, as $\eps\in\mathbb{I}$,
hold:
\begin{enumerate}[leftmargin=*,label=(\roman*),align=left ]
\item \label{enu:2ndBigOh-rifl}$x_{\eps}=O_{\mathcal{J}}(x_{\eps})$;
\item \label{enu:2ndBigOh-trans}if $x_{\eps}=O_{\mathcal{J}}(y_{\eps})$
and $y_{\eps}=O_{\mathcal{J}}(z_{\eps})$ then $x_{\eps}=O_{\mathcal{J}}(z_{\eps})$;
\item \label{enu:2ndBigOh-prod}$O_{\mathcal{J}}(x_{\eps})\cdot O_{\mathcal{J}}(y_{\eps})=O_{\mathcal{J}}(x_{\eps}\cdot y_{\eps})$;
\item \label{enu:2ndBigOh-sum}$O_{\mathcal{J}}(x_{\eps})+O_{\mathcal{J}}(y_{\eps})=O_{\mathcal{J}}\left(\left|x_{\eps}\right|+\left|y_{\eps}\right|\right)$;
\item \label{enu:2ndBigOh-prodExt}$x_{\ep}\cdot O_{\mathcal{J}}(y_{\eps})=O_{\mathcal{J}}(x_{\eps}\cdot y_{\eps})$;
\item \label{enu:2ndBigOh-sumEqualSummand}$O_{\mathcal{J}}(x_{\eps})+O_{\mathcal{J}}(x_{\eps})=O_{\mathcal{J}}(x_{\eps})$;
\item \label{enu:2ndBigOh-sumExt}if $x_{\eps},y_{\eps}\ge0$ for all $\eps\in I$,
then $x_{\eps}+O_{\mathcal{J}}(y_{\eps})=O_{\mathcal{J}}(x_{\eps}+y_{\eps})$;
\item \label{enu:2ndBigOh-prodScal1}$\forall k\in\R_{\ne0}:\ O_{\mathcal{J}}(k\cdot x_{\eps})=O_{\mathcal{J}}(x_{\eps})$;
\item \label{enu:2ndBigOh-prodScal2}$\forall k\in\R:\ k\cdot O_{\mathcal{J}}(x_{\eps})=O_{\mathcal{J}}(x_{\eps})$.
\end{enumerate}
\end{thm}
An analogue of Thm. \ref{thm:2ndBigOh-prop} holds also for the first
notion of big-O, i.e.\ for the relation $x_{\eps}=O_{a,A}(y_{\eps})$
as $\eps\in\mathbb{I}$.

The unifying properties of these notions are explained in the following
results:
\begin{cor}
\label{cor:specialAlgebra}Let $\Omega\subseteq\R^{n}$ be an open
set and $(u_{\eps})\in\mathcal{C}^{\infty}(\Omega,\R)$ be a net of
smooth functions. We use the notations of \cite{GKOS} for moderate
and negligible nets related to the special algebra $\gss(\Omega)$,
and the notations of \cite{ToVe08} for similar notions related to
the algebra $\gsh(\Omega)$ of asymptotic functions. Moreover, we
recall that\foreignlanguage{english}{ $\underline{\eps}:=\min\{\text{\emph{diam}}\left(\text{\emph{supp}}(\eps)\right),1\}$,
where $\eps\in\D(\R^{d})$}. Then
\begin{enumerate}[leftmargin=*,label=(\roman*),align=left ]
\item \foreignlanguage{british}{$(u_{\eps})\in\mathcal{E}_{M}^{\srm}(\Omega)$
if and only if
\[
\forall K\Subset\Omega\,\forall\alpha\in\N^{n}\,\exists N\in\N:\ \sup_{x\in K}\left|\partial^{\alpha}u_{\eps}(x)\right|=O(\eps^{-N})\text{ as }\eps\in\mathbb{I}^{\srm};
\]
}
\item $(u_{\eps})\in\mathcal{N}^{\srm}(\Omega)$ if and only if
\[
\forall K\Subset\Omega\,\forall\alpha\in\N^{n}\,\forall m\in\N:\ \sup_{x\in K}\left|\partial^{\alpha}u_{\eps}(x)\right|=O(\eps^{m})\text{ as }\eps\in\mathbb{I}^{\srm};
\]

\item $(u_{\eps})\in\mathcal{M}\left(\mathcal{E}(\Omega)^{\D_{0}}\right)$
if and only if 
\[
\forall K\Subset\Omega\,\forall\alpha\in\N^{n}\,\exists N\in\N:\ \sup_{x\in K}\left|\partial^{\alpha}u_{\eps}(x)\right|=O(\underline{\eps}^{-N})\text{ as }\eps\in\hat{\mathbb{I}};
\]

\item $(u_{\eps})\in\mathcal{N}\left(\mathcal{E}(\Omega)^{\D_{0}}\right)$
if and only if
\[
\forall K\Subset\Omega\,\forall\alpha\in\N^{n}\,\forall m\in\N:\ \sup_{x\in K}\left|\partial^{\alpha}u_{\eps}(x)\right|=O(\underline{\eps}^{m})\text{ as }\eps\in\hat{\mathbb{I}}.
\]

\end{enumerate}
\end{cor}
To arrive at a similar unifying result for the full algebra, we need
the following
\begin{defn}
\label{def:P(Omega)}Let $\Omega\subseteq\R^{n}$ be an open set.
\begin{enumerate}[leftmargin=*,label=(\roman*),align=left ]
\item \foreignlanguage{british}{If $\eps\in\mathcal{A}_{0}$, then $\Omega_{\eps}:=\Omega\cap\left\{ x\in\R^{n}\mid\text{supp}(\eps)\subseteq\Omega-x\right\} $.}
\item $\mathcal{P}^{\erm}(\Omega):=\prod_{\eps\in I^{\erm}}\mathcal{C}^{\infty}(\Omega_{\eps},\R)$.
\item If $g:X\longrightarrow Z^{Y}$ is a map, then $g^{\vee}:(x,y)\in X\times Y\mapsto g(x)(y)\in Z.$
\end{enumerate}
\end{defn}
\noindent We can say that elements of $\mathcal{P}^{\erm}(\Omega)$
are $I^{\erm}$-indexed nets $(u_{\eps})$ such that $u_{\eps}\in\mathcal{C}^{\infty}(\Omega_{\eps},\R)$.
In \cite{GiWu14} it is proved that $\mathcal{P}^{\erm}(\Omega)$
is isomorphic (as diffeological space, and hence also as set) to the
usual space $\mathcal{E}^{\erm}(\Omega)$ (see \cite{GKOS}).
\begin{thm}
\label{thm:fullEquiv}Let $u=(u_{\eps})\in\mathcal{P}^{\erm}(\Omega)$,
then
\begin{enumerate}[leftmargin=*,label=(\roman*),align=left ]
\item \foreignlanguage{british}{\label{enu:fullEquiv-moderate}$u^{\vee}\in\mathcal{E}_{M}^{\erm}(\Omega)$
if and only if
\[
\forall K\Subset\Omega\,\forall\alpha\in\N^{n}\,\exists N\in\N:\ \sup_{x\in K}\left|\partial^{\alpha}u_{\eps}(x)\right|=O\left(\underline{\eps}^{-N}\right)\text{ as }\eps\in\mathbb{I}^{\erm};
\]
}
\item \label{enu:fullEquiv-negligible}$u^{\vee}\in\mathcal{N}^{\erm}(\Omega)$
if and only if
\[
\forall K\Subset\Omega\,\forall\alpha\in\N^{n}\,\forall m\in\N:\ \sup_{x\in K}\left|\partial^{\alpha}u_{\eps}(x)\right|=O\left(\underline{\eps}^{m}\right)\text{ as }\eps\in\mathbb{I}^{\erm}.
\]

\end{enumerate}
\end{thm}
The same unifying and simple formulation can be used for the diffeomorphism
invariant algebra $\mathcal{G}^{\drm}$, with only one difference:
in the definition of moderate net we use a big-O relation of the type
$O_{\mathcal{J}}$ for a suitable $\mathcal{J}\subset\mathcal{I}$
(see \cite{GiNi14}). For this reason, we are strongly convinced that
the following results can be generalized also to $\mathcal{G}^{\drm}$.

\section{Asymptotic Gauges}

In this section, we are going to introduce some notions for a set
of indices which permit to define what an asymptotic gauge is.

\subsection{\label{sub:ForEpsSufficientlySmall}``For $\eps$ sufficiently small''
in a set of indices}

We start by introducing a useful notation which corresponds, in the
set of indices for the special algebra $\mathbb{I}^{\srm}$, to the
usual ``for $\eps$ sufficiently small''.
\begin{defn}
\label{def:epsSuffSmall}Let $\mathbb{I}=(I,\le,\mathcal{I})$ be
a set of indices. Let $a\in A\in\mathcal{I}$ and $\mathcal{P}(-)$
be a property, then we say
\[
\forall^{\mathbb{I}}\eps\in A_{\le a}:\ \mathcal{P}(\eps),
\]
and we read it \emph{for $\eps$ sufficiently small in $A_{\le a}$
the property $\mathcal{P}(\eps)$ holds}, if
\begin{equation}
\exists e\le a\,\forall\eps\in A_{\le e}:\ \mathcal{P}(\eps).\label{eq:defSuffSmall}
\end{equation}
Moreover, we say that
\[
\forall^{\mathbb{I}}\eps:\ \mathcal{P}(\eps),
\]
and we read it \emph{for $\eps$ sufficiently small in $\mathbb{I}$
the property $\mathcal{P}(\eps)$ holds, if $\exists A\in\mathcal{I}\,\forall a\in A\,\forall^{\mathbb{I}}\eps\in A_{\le a}:\ \mathcal{P}(\eps)$.}\end{defn}
\begin{example}
\label{exa:epsSuffSmall}\ 
\begin{enumerate}[leftmargin=*,label=(\roman*),align=left ]
\item \label{enu:equivEspSuffSmall}By condition \ref{enu:DefSoI-DonwDir}
of Def. \ref{def:setOfIndices} of set of indices, it follows that
$A_{\le e}\ne\emptyset$. Therefore, \eqref{eq:defSuffSmall} is equivalent
to
\[
\exists e\in A_{\le a}\,\forall\eps\in A_{\le e}:\ \mathcal{P}(\eps).
\]
Analogously, we can reformulate similar properties we will see below.
\item We have $x_{\eps}=O_{a,A}(y_{\eps})$ if and only if $\exists H\in\R_{>0}\,\forall^{\mathbb{I}}\eps\in A_{\le a}:\ |x_{\eps}|\le H\cdot|y_{\eps}|$.
\item In the set of indices $\mathbb{I}^{\srm}$, the following properties
are equivalent:

\begin{enumerate}
\item $\exists\eps_{0}\in I^{\srm}\,\forall\eps\in(0,\eps_{0}]:\ \mathcal{P}(\eps)$;
\item $\exists A\in\mathcal{I}^{\srm}\,\exists a\in A\,\forall^{\mathbb{I}^{\srm}}\eps\in A_{\le a}:\ \mathcal{P}(\eps)$;
\item $\exists A\in\mathcal{I}^{\srm}\,\forall a\in A\,\forall^{\mathbb{I}^{\srm}}\eps\in A_{\le a}:\ \mathcal{P}(\eps)$;
\item $\forall A\in\mathcal{I}^{\srm}\,\exists a\in A\,\forall^{\mathbb{I}^{\srm}}\eps\in A_{\le a}:\ \mathcal{P}(\eps)$.
\end{enumerate}
\item In the set of indices $\hat{\mathbb{I}}$, we recall that a property
$\mathcal{P}(\eps)$ is said to hold \emph{almost everywhere} iff
$\left\{ \eps\in\D_{0}\mid\mathcal{P}(\eps)\right\} \in\hat{\mathcal{I}}$
(see \cite{ToVe08}). Using this language, the following properties
are equivalent:

\begin{enumerate}
\item $\mathcal{P}(\eps$) holds almost everywhere;
\item $\exists A\in\hat{\mathcal{I}}\,\exists a\in A\,\forall^{\hat{\mathbb{I}}}\eps\in A_{\le a}:\ \mathcal{P}(\eps)$;
\item $\exists A\in\hat{\mathcal{I}}\,\forall a\in A\,\forall^{\hat{\mathbb{I}}}\eps\in A_{\le a}:\ \mathcal{P}(\eps)$.
\end{enumerate}
\item In the set of indices $\mathbb{I}^{\erm}$, assume that $\phi\in\mathcal{A}_{q}$,
then the following properties are equivalent

\begin{enumerate}
\item $\forall^{\mathbb{I}^{\erm}}\eps\in\left(\mathcal{A}_{q}\right)_{\le\phi}:\ \mathcal{P}(\eps)$;
\item $\exists r\in(0,1]\,\forall s\in(0,r]:\ \mathcal{P}(s\odot\phi)$.
\end{enumerate}
\end{enumerate}
\end{example}

\subsection{Order relation in a set of indices}

All the scales $(\eps^{-n})_{\eps,n}$ of the special algebra are
\emph{positive functions}. Of course, if we change the function $\eps\mapsto\eps^{-n}$,
only for $\eps>\eps_{0}$, so that it is not globally positive anymore,
this will not change anything in the definition of $\gss(\Omega)$.
In this section, we are going to define this order relation for functions
of the type $I\ra\R$.
\begin{defn}
\label{def:order}Let $\mathbb{I}=(I,\le,\mathcal{I})$ be a set of
indices, and $i$, $j:I\ra\R$ be maps. Then we say $i>_{\mathbb{I}}j$
if
\[
\forall^{\mathbb{I}}\eps:\ i_{\eps}>j_{\eps}.
\]

\end{defn}
Following the intuitive interpretation given in \cite{GiNi14}, we
can say that $i>_{\mathbb{I}}j$ if we can find an accuracy class
$A\in\mathcal{I}$ such that for each measuring instrument $a\in A$
in that class, we have $i_{\eps}>j_{\eps}$ for $\eps\in A_{\le a}$
sufficiently small.
\begin{thm}
\label{thm:propOrd}Let $\mathbb{I}=(I,\le,\mathcal{I})$ be a set
of indices, and $i$, $j$, $k$, $z:I\ra\R$ be maps. Then we have:
\begin{enumerate}[leftmargin=*,label=(\roman*),align=left ]
\item \label{enu:noRiflexOrd}$i\not>_{\mathbb{I}}i$.
\item \label{enu:transOrd}If $i>_{\mathbb{I}}j>_{\mathbb{I}}k$, then $i>_{\mathbb{I}}k$.
\item \label{enu:productOrd}If $i>_{\mathbb{I}}0$ and $k>_{\mathbb{I}}j$,
then $i\cdot k>_{\mathbb{I}}i\cdot j$.
\item \label{enu:sumOrd}If $i>_{\mathbb{I}}j$ and $k>_{\mathbb{I}}z$,
then $i+k>_{\mathbb{I}}j+z$.
\item \label{enu:orderBigO}If $\exists A\in\mathcal{I}\,\forall a\in A\,\forall^{\mathbb{I}}\eps\in A_{\le a}:\ i_{\eps}\le j_{\eps}$,
then $i_{\eps}=O(j_{\eps})$ as $\eps\in\mathbb{I}$.
\end{enumerate}
\end{thm}
\begin{proof}
\ref{enu:noRiflexOrd}: By contradiction, assume that $i>_{\mathbb{I}}i$,
i.e.
\begin{equation}
\exists A\in\mathcal{I}\,\forall a\in A\,\forall^{\mathbb{I}}\eps\in A_{\le a}:\ i_{\eps}>i_{\eps}.\label{eq:absNoReflex}
\end{equation}
But $\exists a\in A$ since $\emptyset\notin\mathcal{I}$. This, \eqref{eq:absNoReflex}
and \ref{enu:equivEspSuffSmall} of Example \ref{exa:epsSuffSmall}
yield that for some $e\in A_{\le a}$ we have $i_{\eps}>i_{\eps}$
for all $\eps\in A_{\le e}$. This yields the contradiction $i_{e}>i_{e}$.

\noindent \ref{enu:transOrd}: We can write the assumptions of this
claim as
\begin{align}
\exists A & \in\mathcal{I}\,\forall a\in A\,\forall^{\mathbb{I}}\eps\in A_{\le a}:\ i_{\eps}>j_{\eps}\label{eq:iGreater-j}\\
\exists B & \in\mathcal{I}\,\forall b\in B\,\forall^{\mathbb{I}}\eps\in B_{\le b}:\ j_{\eps}>k_{\eps}.\label{eq:jGreater-k}
\end{align}
By \ref{enu:DefSoI-intersection} we get the existence of $C\in\mathcal{I}$
which is contained in $A\cap B$, i.e.\ where both \eqref{eq:iGreater-j}
and \eqref{eq:jGreater-k} hold. Fix a generic $c\in C$. Both the
relations $i_{\eps}>j_{\eps}$ and $j_{\eps}>k_{\eps}$ hold for $\eps$
sufficiently small in $A_{\le c}$ and $B_{\le c}$ respectively.
Hence
\begin{align}
\exists e' & \le c\,\forall\eps\in A_{\le e'}:\ i_{\eps}>j_{\eps}\label{eq:iGreater-jSuffSmall}\\
\exists e'' & \le c\,\forall\eps\in A_{\le e''}:\ j_{\eps}>k_{\eps}.\label{eq:jGreater-kSuffSmall}
\end{align}
But $(\emptyset,c]$ is downward directed so that there exists $e\le c$
such that $e<e'$ and $e<e''$. For each $\eps\in C_{\le e}$, from
\eqref{eq:iGreater-jSuffSmall} and \eqref{eq:jGreater-kSuffSmall}
we hence get the conclusion $i_{\eps}>j_{\eps}>k_{\eps}$.

\noindent Properties \ref{enu:productOrd} and \ref{enu:sumOrd} can
be proved analogously. Property \ref{enu:orderBigO} also follows
directly from the definitions.\end{proof}
\begin{example}
\label{exa:orderInSoI}\ 
\begin{enumerate}[leftmargin=*,label=(\roman*),align=left ]
\item In the set of indices $\mathbb{I}^{\srm}$, we have $i>_{\mathbb{I}^{\erm}}j$
if and only if $i_{\eps}>j_{\eps}$ for $\eps$ sufficiently small.
\item In the set of indices $\mathbb{I}^{\erm}$ we have $i>_{\mathbb{I}^{\erm}}j$
if and only if there exists $q\in\N$ such that for each $\phi\in\mathcal{A}_{q}$
we have $i(\eps\odot\phi)>j(\eps\odot\phi)$ for $\eps\in(0,1]$ sufficiently
small.
\item In the set of indices $\hat{\mathbb{I}}$ we have $i>_{\hat{\mathbb{I}}}j$
if and only if $i_{\eps}>j_{\eps}$ almost everywhere.
\end{enumerate}
\end{example}

\subsection{Limits in a set of indices}

All the scales $(\eps^{-n})_{\eps,n}$ of the special algebra have
limit $+\infty$ for $\eps\to0^{+}$. In this section, we want to
define the notion of limit in a set of indices for functions of the
type $I\ra\R$. This notion can be easily generalized to generic $f:I\ra T$,
where $T$ is a topological space.
\begin{defn}
\label{def:limit}Let $\mathbb{I}=(I,\le,\mathcal{I})$ be a set of
indices, $f:I\ra\R$ a map, and $l\in\R\cup\{+\infty,-\infty,\infty\}$.
Then we say that \emph{$l$ is the limit of $f$ in $\mathbb{I}$}
if
\begin{equation}
\exists A\in\mathcal{I}\,\forall a\in A:\ l=\lim_{\eps\le a}f|_{A}(\eps),\label{eq:limitDef}
\end{equation}
where the limit \eqref{eq:limitDef} is taken in the downward directed
set $(\emptyset,a]$.\end{defn}
\begin{rem*}
\label{rem:firstPropLimit}\ 
\begin{enumerate}[leftmargin=*,label=(\roman*),align=left ]
\item Writing $\lim_{\eps\le a}f|_{A}(\eps)$ we mean (in case $l$ is
finite)
\begin{equation}
\forall r\in\R_{>0}\,\exists a_{0}\le a\,\forall\eps\in A:\ \eps\le a_{0}\then\left|l-f_{\eps}\right|<r,\label{eq:limitDefExpl}
\end{equation}
that is
\begin{equation}
\forall r\in\R_{>0}\,\forall^{\mathbb{I}}\eps\in A_{\le a}:\ \left|l-f_{\eps}\right|<r.\label{eq:limitDefWithEpsSmall}
\end{equation}
From this, the following properties easily follow

\begin{itemize}
\item If $l=\lim_{\eps\le a}f|_{A}(\eps)$ and $B\subseteq A$, $B\in\mathcal{I}$,
then $l=\lim_{\eps\le a}f|_{B}(\eps)$.
\item There exists at most one $l$ verifying \eqref{eq:limitDefExpl}.
\end{itemize}
\item Let us assume that $l_{1}$ and $l_{2}$ are both limits of $f$ in
$\mathbb{I}$. So, for some $A$, $B\in\mathcal{I}$ we have $l_{1}=\lim_{\eps\le a}f|_{A}(\eps)$
for all $e\in A$, and $l_{2}=\lim_{\eps\le e}f|_{B}(\eps)$ for all
$e\in B$. We can always find $C\in\mathcal{I}$ and $e\in C$ such
that $C\subseteq A\cap B$, so that $l_{1}=\lim_{\eps\le e}f|_{C}(\eps)=l_{2}$.
Therefore, if this limit exists, it is unique and we can use the notation
\[
l=\lim_{\mathbb{I}}f=\lim_{\eps\in\mathbb{I}}f_{\eps}.
\]

\end{enumerate}
\end{rem*}
\begin{example}
\label{exa:limits}\ 
\begin{enumerate}[leftmargin=*,label=(\roman*),align=left ]
\item In the set of indices $\mathbb{I}^{\srm}$, we have $l=\lim_{\mathbb{I}^{\srm}}f$
if and only if $l=\lim_{\eps\to0^{+}}f_{\eps}$.
\item In the set of indices $\mathbb{I}^{\erm}$, we have $l=\lim_{\mathbb{I}^{\erm}}f$
if and only if there exists $q\in\N$ such that for each $\phi\in\mathcal{A}_{q}$
we have $l=\lim_{\eps\to0^{+}}f(\eps\odot\phi)$.
\item In the set of indices $\hat{\mathbb{I}}$, assume that $\hat{\mathcal{I}}$
is a P-point (\cite{CKKR}), and denote by $\hyperR$ the hyperreals
constructed as the ultrapower $\R^{I}/\hat{\mathcal{I}}$. Then we
have

\begin{enumerate}
\item $\exists\lim_{\hat{\mathbb{I}}}f=l\in\R$ if and only if $f$ is finite
and $l$ is the standard part of $[f]_{\hat{\mathcal{I}}}\in\hyperR$.
\item $\exists\lim_{\hat{\mathbb{I}}}f=+\infty$ if and only if $f>0$ and
$[f]_{\hat{\mathcal{I}}}$ is infinite.
\item $\exists\lim_{\hat{\mathbb{I}}}f=-\infty$ if and only if $f<0$ and
$[f]_{\hat{\mathcal{I}}}$ is infinite.
\item $\exists\lim_{\hat{\mathbb{I}}}f=\infty$ if and only if $[\left|f\right|]_{\hat{\mathcal{I}}}$
is infinite.
\end{enumerate}
\end{enumerate}
\end{example}
An expected result is the following
\begin{thm}
\label{thm:limitProp}Let $\mathbb{I}=(I,\le,\mathcal{I})$ be a set
of indices and $f$, $g$, $h:I\ra\R$ be maps such that the limits
$\lim_{\mathbb{I}}f$, $\lim_{\mathbb{I}}g$ exists and are finite.
Then
\begin{enumerate}[leftmargin=*,label=(\roman*),align=left ]
\item $\exists\lim_{\mathbb{I}}(f+g)=\lim_{\mathbb{I}}f+\lim_{\mathbb{I}}g$.
\item $\exists\lim_{\mathbb{I}}(f\cdot g)=\lim_{\mathbb{I}}f\cdot\lim_{\mathbb{I}}g$.
\item $\forall r\in\R:\ \exists\lim_{\mathbb{I}}r=r$.
\item If $\lim_{\mathbb{I}}g\ne0$, then $\exists\lim_{\mathbb{I}}\frac{f}{g}=\frac{\lim_{\mathbb{I}}f}{\lim_{\mathbb{I}}g}$.
\item If $f<_{\mathbb{I}}h<_{\mathbb{I}}g$ and $\lim_{\mathbb{I}}f=\lim_{\mathbb{I}}g=:l$,
then $\exists\lim_{\mathbb{I}}h=l$.
\item \label{enu:limitOrder}If $\exists\lim_{\mathbb{I}}f>0$, then $f>_{\mathbb{I}}0$.
\end{enumerate}
\end{thm}
The proof is a direct consequence of our definition of limit and,
as usual, property \ref{enu:DefSoI-intersection} of Def. \ref{def:setOfIndices}.

\subsection{\label{sub:AsymptoticGauges}Asymptotic gauges}

Asymptotic gauges represent our definition of scale for Colombeau
like algebras. The idea is, essentially, to ask the \emph{asymptotic}
closure of $\mathcal{B}\subseteq\R^{I}$ with respect to algebraic
operations.
\begin{defn}
\label{def:asymptoticGauge}Let $\mathbb{I}=(I,\le,\mathcal{I})$
be a set of indices. All the big-O in this definition have to be meant
as $O_{\mathbb{I}}$ (see Def. \ref{def:2ndBigOh}). Then we say that
$\mathcal{B}$ is an \emph{asymptotic gauge on $\mathbb{I}$} (briefly:
AG on $\mathbb{I}$) if
\begin{enumerate}[leftmargin=*,label=(\roman*),align=left ]
\item \label{enu:AGnets}$\mathcal{B}\subseteq\R^{I}$;
\item \label{enu:AGinfinite}$\exists i\in\mathcal{B}:\ \lim_{\mathbb{I}}i=\infty$;
\item \label{enu:AGProduct}$\forall i,j\in\mathcal{B}\,\exists p\in\mathcal{B}:\ i\cdot j=O(p)$;
\item \label{enu:AGproductScalar}$\forall i\in\mathcal{B}\,\forall r\in\R\,\exists\sigma\in\mathcal{B}:\ r\cdot i=O(\sigma)$;
\item \label{enu:AGabsSum}$\forall i,j\in\B\,\exists s\in\B:\ s>_{\mathbb{I}}0\ ,\ |i|+|j|=O(s)$.
\end{enumerate}

Moreover, we say that:
\begin{itemize}
\item $\mathcal{B}>_{\mathbb{I}}0$, and we read it as $\mathcal{B}$ \emph{is
positive}, if $i>_{\mathbb{I}}0$ for each $i\in\mathcal{B}$;
\item $\mathcal{B}$ \emph{is totally ordered} if for all $i$, $j\in\mathcal{B}$
either $i=O(j)$ or $j=O(i)$;
\item $\B_{>0}:=\{i\in\B\mid i>_{\mathbb{I}}0\}$.
\end{itemize}
\end{defn}
Of course, any solid subalgebra (see \cite[Def. 9]{Has11}) of $\R^{I}$
containing at least an infinite net is a trivial asymptotic gauge.
To include this case, we only asked an existence in \ref{enu:AGinfinite}
of the previous definition.

The name \emph{gauge} gives the idea that using the infinities of
$\mathcal{B}$, we are going to define moderate nets for our Colombeau-like
algebras. To define negligible nets, we can use $b^{-1}$ for all
$b\in\mathcal{B}$ or another asymptotic gauge $\mathcal{Z}$ ``at
least as strong as $\mathcal{B}$''.

The first example corresponds, of course, to the special algebra and
so it starts from the set of indices $\mathbb{I}^{\srm}$ and it is
defined as $\mathcal{B}^{\srm}:=\left\{ \left(\eps^{-a}\right)\mid a\in\R_{>0}\right\} $.
This AG is positive and totally ordered.

Property \ref{enu:AGabsSum} of Def. \ref{def:asymptoticGauge} is
equivalent to ask for the asymptotic closure of $\B$ with respect
to sum and to absolute value, as it is stated in the following
\begin{lem}
\label{lem:AGabsSum}Let $\B\subseteq\R^{I}$, then property \ref{enu:AGabsSum}
of Def. \ref{def:asymptoticGauge} is equivalent to
\begin{enumerate}[leftmargin=*,label=(\roman*),align=left ]
\item \label{enu:AGsum}$\forall i,j\in\B\,\exists s\in\B:\ i+j=O(s)$;
\item \label{enu:AGabs}$\forall i\in\B\,\exists a\in\B:\ a>_{\mathbb{I}}0\ ,\ i=O(a)$.
\end{enumerate}
\end{lem}
\begin{proof}
That Def. \ref{def:asymptoticGauge}.\ref{enu:AGabsSum} is sufficient
follows from $i+j\le|i|+|j|$, $i\le|i|+|i|$ and Thm. \ref{thm:propOrd}.\ref{enu:orderBigO}.
The condition is also necessary: if $i$, $j\in\B$, then from \ref{enu:AGabs}
we get $a$, $b\in\B_{>0}$ such that $i=O(a)$ and $j=O(b)$; from
\ref{enu:AGsum} we have the existence of $s\in\B$ such that $a+b=O(s)$.
Once again from \ref{enu:AGabs}, we can assume $s>_{\mathbb{I}}0$.
Thus $|i|+|j|=O(|a|+|b|)=O(a+b)=O(s)$.
\end{proof}
A general way to obtain an asymptotic gauge on a generic set of indices
$\mathbb{I}$ is to find a map $\rho:I\ra(0,1]$ such that $\lim_{\mathbb{I}}\rho=0$
and to take the composition of $\rho$ with the nets of an asymptotic
gauge $\mathcal{B}$ on $\mathbb{I}^{\srm}$, i.e.
\[
\mathcal{B}\circ\rho:=\left\{ i\circ\rho\mid i\in\mathcal{B}\right\} .
\]
For example, for the set of indices $\mathbb{I}^{\erm}$ we can consider
\begin{equation}
\rho(\phi):=\min\left\{ \text{diam}\left(\text{supp}\phi\right),1\right\} \quad\forall\phi\in I^{\erm}=\mathcal{A}_{0}.\label{eq:defRhoInf}
\end{equation}
An analogue function $\rho$ can be defined for the set of indices
$\hat{\mathbb{I}}$ and in both cases they have limit zero. Several
of the definitions we have introduced so far are motivated by the
wish to obtain the following
\begin{thm}
\label{thm:AGfromInfinitesRho}Let $\mathbb{I}=(I,\le,\mathcal{I})$
be a set of indices, and $\rho:I\ra(0,1]$ be a map such that $\lim_{\mathbb{I}}\rho=0$.
Let $\mathcal{B}$ be an asymptotic gauge on $\mathbb{I}^{\srm}$,
then
\begin{enumerate}[leftmargin=*,label=(\roman*),align=left ]
\item \label{enu:BcompRhoAG}$\mathcal{B}\circ\rho$ is an asymptotic gauge
on $\mathbb{I}$.
\item \label{enu:BcompRhoPositive}If $\mathcal{B}>_{\mathbb{I}^{\srm}}0$
then $\mathcal{B}\circ\rho>_{\mathbb{I}}0$.
\item \label{enu:BcompRhoTotOrd}If $\mathcal{B}$ is totally ordered, then
also $\mathcal{B}\circ\rho$ is totally ordered.
\end{enumerate}
\end{thm}
\begin{proof}
Property \ref{enu:AGnets} of Def. \ref{def:asymptoticGauge} is clear.
To prove \ref{enu:AGinfinite} of Def. \ref{def:asymptoticGauge},
assume that $i\in\mathcal{B}$ is such that $\lim_{r\to0^{+}}i_{r}=\infty$.
Since $\lim_{\mathbb{I}}\rho=0$, we get
\begin{equation}
\exists A\in\mathcal{I}\,\forall a\in A:\ 0=\lim_{\eps\le a}\rho|_{A}(\eps).\label{eq:limitRhoExplicit}
\end{equation}
For each $R\in\R_{>0}$ there exists $\delta\in\R_{>0}$ such that
$|i_{r}|>R$ for $r\in(0,\delta]$. But for each $a\in A$, \ref{eq:limitRhoExplicit}
yields
\[
\exists\eps_{0}\le e\,\forall\eps\in A:\ \eps\le\eps_{0}\then|\rho_{\eps}|<\delta.
\]
Therefore, $|i(\rho_{\eps})|>r$ for the same $\eps\in A_{\le\eps_{0}}$.
This proves that $\lim_{\mathbb{I}}i(\rho_{\eps})=\infty$.

Take $i$, $j\in\mathcal{B}$. We want to prove that both $\left|i\circ\rho\right|+\left|j\circ\rho\right|$
and $(i\circ\rho)\cdot(j\circ\rho)$ are asymptotically in $\mathcal{B}\circ\rho$.
From the analogous property of the AG $\mathcal{B}$, we get the existence
of $\sigma$, $\pi\in\mathcal{B}$ such that $|i|+|j|=O_{\mathbb{I}^{\srm}}(\sigma)$
and $i\cdot j=O_{\mathbb{I}^{\srm}}(\pi)$, that is $||i_{r}|+|j_{r}||\le H\cdot|\sigma_{r}|$
and $|i_{r}\cdot j_{r}|\le K\cdot|\pi_{r}|$ for some $H$, $K\in\R_{>0}$
and for each $r\in(0,r_{0}]$. It suffices to take $A$ as in \ref{eq:limitRhoExplicit}
and $a\in A$ to have $|\rho_{\eps}|<r_{0}$ for each $\eps\in A_{\le\eps_{1}}$,
for a suitable $\eps_{1}\le a$. For all these $\eps\in A_{\le\eps_{1}}$
we hence get $\left|\left|i(\rho_{\eps})\right|+\left|j(\rho_{\eps})\right|\right|\le H\cdot|\sigma(\rho_{\eps})|$
and $|i(\rho_{\eps})\cdot j(\rho_{\eps})|\le K\cdot|\pi(\rho_{\eps})|$,
which is our conclusion. Analogously, we can prove the asymptotic
closure with respect to the product by scalars.

To prove \ref{enu:BcompRhoPositive}, assume that $i\in\mathcal{B}>_{\mathbb{I}^{\srm}}0$,
so that $i_{r}>0$ for $r\in(0,r_{0}]$. From \eqref{eq:limitRhoExplicit}
for each $a\in A$ we have that $\left|\rho_{\eps}\right|<r_{0}$
for all $\eps\in A_{\le\eps_{1}}$ and for some $\eps_{1}\le a$.
Therefore $i(\rho_{\eps})>0$ for the same $\eps$. This proves that
$i(\rho_{\eps})>0$ $\forall^{\mathbb{I}}\eps\in A_{\le a}$, and
hence also that $i\circ\rho>_{\mathbb{I}}0$.

Finally, to prove \ref{enu:BcompRhoTotOrd}, take $i$, $j\in\mathcal{B}$
such that $i_{r}=O_{\mathbb{I}^{\srm}}(j_{r})$ as $r\in\mathbb{I}^{\srm}$.
We want to prove that $i(\rho_{\eps})=O_{\mathbb{I}}(j(\rho_{\eps}))$
as $\eps\in\mathbb{I}$. Assume that $|i_{r}|\le H\cdot|j_{r}|$ for
$r\in(0,r_{0}]$. The limit relation \eqref{eq:limitRhoExplicit}
yields the existence of $A\in\mathcal{I}$ such that for each $a\in A$
we have
\[
\exists\eps_{0}\le a\,\forall\eps\in A_{\le\eps_{0}}:\ |\rho_{\eps}|<r_{0},
\]
and thus $|i(\rho_{\eps})|\le H\cdot|j(\rho_{\eps})|$ for the same
$\eps$. By \ref{enu:DefSoI-DonwDir} of Def. \ref{def:setOfIndices}
of set of indices, we get $A_{\le\eps_{0}}\ne\emptyset$ which yields
\[
\exists\eps_{1}\in A_{\le a}\,\forall\eps\in A_{\le\eps_{1}}:\ |i(\rho_{\eps})|\le H\cdot|j(\rho_{\eps})|,
\]
which proves that $i(\rho_{\eps})=O_{a,A}(j(\rho_{\eps}))$, that
is our conclusion.\end{proof}
\begin{example}
\label{exa:AG}\ 
\begin{enumerate}[leftmargin=*,label=(\roman*),align=left ]
\item Let $\rho$ be the function defined in \eqref{eq:defRhoInf}, then
$\left\{ \left(\rho_{\eps}^{-a}\right)\mid a\in\R_{>0}\right\} $
is a totally ordered asymptotic gauge of positive functions on $\mathbb{I}^{\erm}$.
In the same way, we can proceed for $\hat{\mathbb{I}}$.
\item Define $\text{exp}(x):=e^{x}$ for $x\in\R$, and $\text{exp}^{k}:=\exp\circ\ptind^{k}\circ\text{ exp}$,
then $\mathcal{B}_{\text{fin}}^{\text{exp}}:=\left\{ \text{exp}^{k}\left(\frac{1}{\eps}\right)\mid k\in\N_{\ne0}\right\} $
and $\mathcal{B}_{\infty}^{\text{exp}}:=\left\{ \left(\text{exp}^{[\eps^{-1}]}(\eps^{-1})^{a}\right)\mid a\in\R_{>0}\right\} $,
where $[x]$ is the integer part of $x\in\R$, are totally ordered
asymptotic gauges of positive functions.
\item Assuming that $\hat{\mathcal{I}}$ is a P-point on $\hat{I}=\D(\R^{d})$,
then the condition $\rho:\hat{I}\ra\R_{>0}$ and $\lim_{\hat{\mathbb{I}}}\rho=0$
are equivalent to say that $[\rho]_{\hat{\mathcal{I}}}\in\hyperR=\R^{\hat{I}}/\hat{\mathcal{I}}$
is infinitesimal. Therefore, Thm. \ref{thm:AGfromInfinitesRho} gives
that $\hat{\mathcal{B}}_{\rho}:=\{\rho^{-a}\mid a\in\R_{>0}\}$, $\hat{\mathcal{B}}_{\text{fin}}^{\text{exp}}:=\{\text{exp}^{k}(\rho^{-1})\mid k\in\N_{\ne0}\}$
and $\hat{\mathcal{B}}_{\infty}^{\text{exp}}:=\left\{ \text{exp}^{[\rho^{-1}]}(\rho^{-1})^{a}\mid a\in\R_{>0}\right\} $
are totally ordered asymptotic gauges of positive functions on $\hat{\mathbb{I}}$.
\end{enumerate}
\end{example}
\begin{defn}
Let $\B$ be an AG on the set of indices $\mathbb{I}=(I,\le,\mathcal{I})$.
The set of \emph{moderate nets generated by} \emph{$\B$} is 

\[
\R_{M}(\B):=\left\{ x\in\R^{I}\mid\exists b\in\B:\ x_{\eps}=O(b_{\eps})\text{ as }\eps\in\mathbb{I}\right\} .
\]

\end{defn}
It is immediate to see that $\B\subseteq\R_{M}(\B)$. Let us introduce
also the following definition:
\begin{defn}
We say that an AG $\B$ is an \emph{asymptotically closed ring} if
\begin{equation}
\B=\R_{M}(\B).\label{eq:asymptClosed}
\end{equation}

\end{defn}
The following holds:
\begin{thm}
\label{thm:acr} If $\B$ is an AG then $\R_{M}(\B)$ is the minimal
(with respect to inclusion) asymptotically closed solid ring containing
$\B$.\end{thm}
\begin{proof}
From Def. \ref{def:asymptoticGauge}.\ref{enu:AGinfinite} and Thm.
\textcolor{red}{\ref{thm:propOrd}} we get $0\in\R_{M}(\B)$. If $x_{\eps}=O(b_{\eps})$
and $y_{\eps}=O(c_{\eps})$ for $b,c\in\B$, then $x_{\eps}+y_{\eps}=O(|b_{\eps}|+|c_{\eps}|)$.
But $|b_{\eps}|+|c_{\eps}|=O(d_{\eps})$ for some $d\in\B_{>0}$ by
Def. \ref{def:asymptoticGauge}.\ref{enu:AGabsSum}. Therefore, $x_{\eps}+y_{\eps}=O(d_{\eps})$.
Analogously, we can prove the closure of $\R_{M}(\B)$ with respect
to the product. It is immediate to see that $\R_{M}(\R_{M}(\B))=\R_{M}(\B)$.
Let us prove the minimality: let $R$ be an asymptotically closed
ring containing $\B$. Let $x\in\R_{M}(\B)$, and let $a\in\B_{>0}$
be such that $x=O(a)$. Then, since $a\in\B\subseteq R$, we get $x\in\R_{M}(R)$.
But $R=\R_{M}(R)$ because $R$ is asymptotically closed, and we get
that $x\in R$. So $\R_{M}(\B)\subseteq R$. The definition of $\R_{M}(\B)$
directly gives that it is also asymptotically solid.\end{proof}
\begin{defn}
Given two asymptotic gauges $\B_{1},\B_{2}$ we say that $\B_{1}$
and $\B_{2}$ are\emph{ equivalent} if $\R_{M}(\B_{1})=\R_{M}(\B_{2})$.\end{defn}
\begin{example}
\ 
\begin{enumerate}[leftmargin=*,label=(\roman*),align=left ]
\item Every AG $\B$ is equivalent to $\R_{M}(\B)$.
\item The asymptotic gauges $\B_{1}=\{(\eps^{-a})\mid a\in\R\}$, $\B_{2}=\{(\eps^{-2a})\mid a\in\R\}$
and $\B_{3}=\{(\eps^{-n})\mid n\in\N\}$ on $I=(0,1]$ are all equivalent.
\end{enumerate}
\end{example}

\section{Colombeau algebras generated by two asymptotic gauges}

As we have already stated, every asymptotic gauge formalizes a notion
of \char`\"{}growth conditions\char`\"{}. For example, $\B^{s}$ (see
Sec.\ \ref{sub:AsymptoticGauges}) formalizes the idea of polynomial
growth. We can hence use an asymptotic gauge $\B$ to define moderate
nets and the reciprocals of nets taken from another asymptotic gauge
$\mathcal{Z}$ to define negligible nets. From this point of view,
it is natural to introduce the following definition:
\begin{defn}
Let $\Omega\subseteq\mathbb{R}^{n}$ be an open set, let $\mathcal{B},\mathcal{Z}$
be AG on a set of indices $\mathbb{I}=(I,\leq,\mathcal{I})$ and let
$\mathcal{A}$ be a subalgebra of $\mathcal{C}^{\infty}(\Omega)^{I}$.
The \emph{set of $\B$-moderate nets in $\mathcal{A}$} is 
\begin{multline*}
\mathcal{E}_{M}(\B,\Omega,\mathcal{A}):=\{u\in\mathcal{A}\mid\forall K\Subset\Omega\,\forall\alpha\in\mathbb{N}^{n}\\
\exists b\in\B:\ \sup\limits _{x\in K}|\partial^{\alpha}u_{\eps}(x)|=O(b_{\eps})\text{ as }\eps\in\mathbb{I}\}.
\end{multline*}

\noindent The \emph{set of $\mathcal{Z}$-negligible nets in $\mathcal{A}$}
is
\begin{multline}
\mathcal{N}(\mathcal{Z},\Omega,\mathcal{A}):=\{u\in\mathcal{A}\mid\forall K\Subset\Omega\,\forall\alpha\in\mathbb{N}^{n}\\
\forall z\in\mathcal{Z}_{>0}:\ \sup\limits _{x\in K}|\partial^{\alpha}u_{\eps}(x)|=O(z_{\eps}^{-1})\text{ as }\eps\in\mathbb{I}\}.\label{eq:DefNegligible}
\end{multline}

\noindent Moreover we set

\[
\mathcal{E}_{M}(\B,\Omega):=\mathcal{E}_{M}(\B,\Omega,\mathcal{C}^{\infty}(\Omega)^{I})
\]

\noindent and

\[
\mathcal{N}(\mathcal{Z},\Omega):=\mathcal{N}(\mathcal{Z},\Omega,\mathcal{C}^{\infty}(\Omega)^{I}).
\]
\end{defn}
\begin{rem*}
\label{rem:DefModerateNegligible}\ 
\begin{enumerate}[leftmargin=*,label=(\roman*),align=left ]
\item If $z\in\mathcal{Z}_{>0}$ then $\forall^{\mathbb{I}}\eps\in A_{\le a}:\ z_{\eps}>0$
for some $a\in A\in\mathcal{I}$, so that we can consider $z_{\eps}^{-1}$.
It is implicit in \eqref{eq:DefNegligible} that we are considering
only these $\eps$.
\item $\mathcal{E}_{M}(\B,\Omega)\cap\R^{I}=\R_{M}(\B)$ (by identifying
a constant function $\Omega\ra\R$ with its value).
\item $\mathcal{E}_{M}(\B,\Omega,\mathcal{A})=\mathcal{E}_{M}(\B,\Omega)\cap\mathcal{A}$.
\item $ $$\mathcal{N}(\mathcal{Z},\Omega,\mathcal{A}):=\mathcal{N}(\mathcal{Z},\Omega)\cap\mathcal{A}$.
\end{enumerate}
\end{rem*}
We want to find conditions that ensure that the quotient $\mathcal{E}_{M}(\B,\Omega)/\mathord{\mathcal{N}(\mathcal{Z},\Omega)}$
is an algebra. When this happens, we will use the following definition:
\begin{defn}
Let $\B,\mathcal{Z}$ be AG and let $\mathcal{A}$ be a subalgebra
of $\mathcal{C}^{\infty}(\Omega)^{I}$. The \emph{Colombeau AG algebra
generated by $\B$ and $\mathcal{Z}$ on $\mathcal{A}$} is the quotient
\[
\mathcal{G}(\B,\mathcal{Z},\mathcal{A}):=\mathcal{E}_{M}(\B,\Omega,\mathcal{A})/\mathord{\mathcal{N}(\mathcal{Z},\Omega,\mathcal{A})}.
\]

\noindent We also set $\mathcal{G}(\B,\mathcal{Z},\mathcal{C}^{\infty}(\Omega)^{I})=\GBZ$.
\end{defn}
In the following, we will only consider the case $\mathcal{A}=\mathcal{C}^{\infty}(\Omega)^{I}$
and real valued nets of smooth functions. Nevertheless, all the results
that we prove in this section can be easily generalized to the case
of a generic subalgebra $\mathcal{A}\subseteq\mathcal{C}^{\infty}(\Omega)^{I}$
and to complex valued nets.

Let us observe that $\mathcal{G}^{s}(\Omega)=\mathcal{G}(\B^{s},\B^{s},\Omega)$.
A known result is that, having fixed $\mathcal{Z}=\B^{s}$, $\B=\mathcal{Z}$
is the maximal choice such that $\GBZ$ is an algebra. We will prove
that a similar property holds in our general setting.
\begin{lem}
\label{lem:negligibleAreModerate}For every $\B,\mathcal{Z}$ asymptotic
gauges and $\Omega\subseteq\mathbb{R}^{n}$ the inclusion $\mathcal{N}(\mathcal{Z},\Omega)\subseteq\mathcal{E}_{M}(\B,\Omega)$
holds.\end{lem}
\begin{proof}
Let $b\in\B$ and $z\in\mathcal{Z}$ be infinite nets: $\lim_{\eps\in\mathbb{I}}b_{\eps}=\lim_{\eps\in\mathbb{I}}z_{\eps}=\infty$.
Then $\lim_{\eps\in\mathbb{I}}z_{\eps}^{-1}=0$, so that, for some
$a\in A\in\mathcal{I}$, $|z_{\eps}^{-1}|<1$ $\forall^{\mathbb{I}}\eps\in A_{\le a}$.
Analogously, $1<|b_{\eps}|$ $\forall^{\mathbb{I}}\eps\in B_{\le b}$
for some $b\in B\in\mathcal{I}$. Therefore, $z_{\eps}^{-1}=O(b_{\eps})$.
Now, let $K\Subset\Omega$, $\alpha\in\N^{n}$ and $u\in\mathcal{N}(\mathcal{Z},\Omega)$.
As 
\[
\sup\limits _{x\in K}|\partial^{\alpha}u_{\eps}(x)|=O(z_{\eps}^{-1})
\]
and $z_{\eps}^{-1}=O(b_{\eps})$, we obtain that 
\[
\sup\limits _{x\in K}|\partial^{\alpha}u_{\eps}(x)|=O(b_{\eps}),
\]
so $\ue\in\mathcal{E}_{M}(\B,\Omega)$.\end{proof}
\begin{lem}
\label{lem:moderateNegligibleRings}For every $\B,\mathcal{Z}$ AG
and $\Omega\subseteq\mathbb{R}^{n}$ open set, both $\mathcal{E}_{M}(\B,\Omega)$
and $\mathcal{N}(\mathcal{Z},\Omega)$ are rings.\end{lem}
\begin{proof}
Let $u,v\in\mathcal{E}_{M}(\B,\Omega)$. Let $K\Subset\Omega$, $\alpha\in\N^{n}$
and let $b,c\in\B_{>0}$ be such that 
\[
\sup\limits _{x\in K}|\partial^{\alpha}u_{\eps}(x)|=O(b_{\eps}),\ \sup\limits _{x\in K}|\partial^{\alpha}v_{\eps}(x)|=O(c_{\eps}).
\]
Finally, let $d\in\B_{>0}$ be such that $|b_{\eps}|+|c_{\eps}|=O(d_{\eps})$.
Then 
\[
\sup\limits _{x\in K}|\partial^{\alpha}(u_{\eps}+v_{\eps})(x)|=O(b_{\eps})+O(c_{\eps})=O(|b_{\eps}|+|c_{\eps}|)=O(d_{\eps}),
\]
so $u+v\in\mathcal{E}_{M}(\B,\Omega)$. Similarly, we can proceed
for the product.

\noindent Now let $u,v\in\mathcal{N}(\mathcal{Z},\Omega)$. Let $K\Subset\Omega$,
$\alpha\in\N^{n}$ and let $z\in\mathcal{Z}_{>0}$. Then 
\[
\sup\limits _{x\in K}|\partial^{\alpha}(u_{\eps}+v_{\eps})(x)|=O(z_{\eps}^{-1})+O(z_{\eps}^{-1})=O(z_{\eps}^{-1}),
\]
so $u+v\in\mathcal{N}(\mathcal{Z},\Omega)$. Similarly, we can proceed
for the product.
\end{proof}
To have that $\GBZ$ is an algebra, we need that the product of a
moderate net by a negligible one is always negligible. This implies
that if $b\in\B$ and $z\in\mathcal{Z}_{>0}$, then we can find a
$w\in\mathcal{Z}_{>0}$, depending on $b$ and sufficiently small,
such that $w_{\eps}^{-1}\cdot b_{\eps}$ is bounded by $z_{\eps}^{-1}$.
This forces a relation between the AG $\B$ and $\mathcal{Z}$ which
can be summarized by saying that the scale $\mathcal{Z}$ is stronger
or equal to that of $\B$, as it is precisely stated in the following
theorem.
\begin{thm}
\noindent \label{thm:equivCondForIdeal}Let $\B$, $\mathcal{Z}$
be AG and $\Omega\subseteq\R^{n}$ be an open set. Then the following
properties are equivalent
\begin{enumerate}[leftmargin=*,label=(\roman*),align=left ]
\item \label{enu:bigOhInclusion}$\R_{M}(\B)\subseteq\R_{M}(\mathcal{Z})$;
\item \label{enu:B-Z-relByRepresentatives}$\forall b\in\B\,\forall z\in\mathcal{Z}_{>0}\,\exists w\in\mathcal{Z}_{>0}:\ w_{\eps}^{-1}\cdot b_{\eps}=O(z_{\eps}^{-1})$.
\end{enumerate}

\noindent If these hold, then
\begin{enumerate}[leftmargin=*,label=(\roman*),align=left,start=3]
\item \label{enu:ideal}$\mathcal{N}(\mathcal{Z},\Omega)$ is a multiplicative
ideal in $\mathcal{E}_{M}(\B,\Omega)$ (so, in particular, the quotient
$\GBZ$ is an algebra). 
\end{enumerate}

\noindent Moreover, if
\begin{equation}
\forall b\in\B\,\forall z\in\mathcal{Z}:\ b_{\eps}=O(z_{\eps})\text{ or }z_{\eps}=O(b_{\eps})\label{eq:B-Z_O-totallyOrdered}
\end{equation}
then \ref{enu:ideal} entails \ref{enu:bigOhInclusion}.

\end{thm}
\begin{proof}
\ref{enu:bigOhInclusion} $\Rightarrow$ \ref{enu:B-Z-relByRepresentatives}:
Let $b\in\B$ and $z\in\mathcal{Z}_{>0}$, we have to prove that $z\cdot b\in\R_{M}(\mathcal{Z})$.
But $b\in\B\subseteq\R_{M}(\B)\subseteq\R_{M}(\mathcal{Z})$ and $z\in\mathcal{Z}_{>0}\subseteq\R_{M}(\mathcal{Z})$.
Since $\R_{M}(\mathcal{Z})$ is a ring, the conclusion follows.

\noindent \ref{enu:B-Z-relByRepresentatives} $\Rightarrow$ \ref{enu:bigOhInclusion}:
Take $x\in\R_{M}(\B)$, so that $x_{\eps}=O(b_{\eps})$ for some $b\in\B$.
Take an infinite net $z\in\mathcal{Z}$: $\lim_{\eps\in\mathbb{I}}z_{\eps}=\infty$.
Then $b_{\eps}=O(z_{\eps}\cdot b_{\eps})$ and hence $x_{\eps}=O(z_{\eps}\cdot b_{\eps})$.
By \ref{enu:B-Z-relByRepresentatives} we get $w\in\mathcal{Z}_{>0}$
such that $z_{\eps}\cdot b_{\eps}=O(w_{\eps})$, which implies $x\in\R_{M}(\mathcal{Z})$.

\noindent \ref{enu:B-Z-relByRepresentatives} $\Rightarrow$ \ref{enu:ideal}:
Let $u\in\mathcal{N}(\mathcal{Z},\Omega)$ and $v\in\mathcal{E}_{M}(\B,\Omega)$,
let $K\Subset\Omega$, $\alpha\in\mathbb{N}^{n}$. Since
\[
\sup_{x\in K}\left|\partial^{\alpha}(u_{\eps}\cdot v_{\eps})(x)\right|\le\sum_{\substack{\beta\in\N^{n}\\
\beta\le\alpha
}
}{\alpha \choose \beta}\sup_{x\in K}\left|\partial^{\beta}u_{\eps}(x)\right|\cdot\sup_{x\in K}\left|\partial^{\alpha-\beta}v_{\eps}(x)\right|\quad\forall\eps\in I,
\]
for each $\beta\le\alpha$ we can find $b_{\beta}=(b_{\beta\eps})\in\B_{>0}$
such that 
\[
\sup\limits _{x\in K}|\partial^{\alpha-\beta}v_{\eps}(x)|=O(b_{\beta\eps}).
\]
Therefore, for all $w\in\mathcal{Z}_{>0}$ we obtain
\begin{align}
\sup_{x\in K}\left|\partial^{\alpha}(u_{\eps}\cdot v_{\eps})(x)\right| & =\sum_{\substack{\beta\in\N^{n}\\
\beta\le\alpha
}
}{\alpha \choose \beta}O(w_{\eps}^{-1})O(b_{\beta\eps})=\nonumber \\
 & =O\left(w_{\eps}^{-1}\right)\cdot O\left(\sum_{\substack{\beta\in\N^{n}\\
\beta\le\alpha
}
}b_{\beta\eps}\right).\label{eq:supDerProd}
\end{align}
Let $b\in\B_{>0}$ such that $\sum_{\substack{\beta\in\N^{n}\\
\beta\le\alpha
}
}b_{\beta\eps}=O(b_{\eps})$ and $z\in\mathcal{Z}_{>0}$. By \ref{enu:B-Z-relByRepresentatives}
there exists $(w_{\eps})\in\mathcal{Z}_{>0}$ such that $w_{\eps}^{-1}\cdot b_{\eps}=O(z_{\eps}^{-1})$,
and hence $\sup_{x\in K}\left|\partial^{\alpha}(u_{\eps}\cdot v_{\eps})(x)\right|=O(z_{\eps}^{-1})$
by \eqref{eq:supDerProd}.

\noindent \ref{enu:ideal} and \eqref{eq:B-Z_O-totallyOrdered} $\Rightarrow$
\ref{enu:bigOhInclusion}: Let us assume that \ref{enu:bigOhInclusion}
does not hold. Then there exists $x\in\R_{M}(\B)\setminus\R_{M}(\mathcal{Z})$,
i.e.\ $x_{\eps}=O(b_{\eps})$ for some $b\in\B_{>0}$. Since $x\notin\R_{M}(\mathcal{Z})$,
it is easy to see that $b\notin\R_{M}(\mathcal{Z})$. From $b\notin\R_{M}(\mathcal{Z})$,
we get $b_{\eps}\neq O(z_{\eps})$ for all $(z_{\eps})\in\mathcal{Z}$.
Thus, \eqref{eq:B-Z_O-totallyOrdered} yields $z_{\eps}=O(b_{\eps})$
for all $(z_{\eps})\in\mathcal{Z}$. Hence $b_{\eps}^{-1}=O(z_{\eps}^{-1})$
for all $(z_{\eps})\in\mathcal{Z}_{>0}$, so we have that
\[
(b_{\eps})\in\mathcal{E}_{M}(\B,\Omega);
\]
\[
(b_{\eps})^{-1}\in\mathcal{N}(\mathcal{Z},\Omega);
\]
\[
(b_{\eps})\cdot(b_{\eps}^{-1})\notin\mathcal{N}(\mathcal{Z},\Omega),
\]
so $\mathcal{N}(\mathcal{Z},\Omega)$ is not a multiplicative ideal
in $\mathcal{E}_{M}(\B,\Omega)$, which is absurd.
\end{proof}
Let us note that, in particular, when \eqref{eq:B-Z_O-totallyOrdered}
holds, $\GBZ$ is an algebra if and only if $\R_{M}(\B)\subseteq\R_{M}(\mathcal{Z})$,
so the maximal possible choice for $\B$ is to take $\B$ equivalent
to $\mathcal{Z}$. This is the generalization to our context of the
known result for $\mathcal{G}^{s}(\Omega)$.

We conclude this section by proving that equivalent asymptotic gauges
give the same Colombeau AG algebras:
\begin{thm}
\label{thm:ograndi} Let $\B,\mathcal{Z}$ be AG, and let $\Omega\subseteq\mathbb{R}^{n}$.
The following conditions hold: 
\begin{enumerate}[leftmargin=*,label=(\roman*),align=left ]
\item $\mathcal{E}_{M}(\B,\Omega)=\mathcal{E}_{M}(\R_{M}(\B),\Omega);$
\item $\mathcal{N}(\mathcal{Z},\Omega)=\mathcal{N}(\R_{M}(\mathcal{Z}),\Omega)$. 
\end{enumerate}

In particular, $\mathcal{G}(\B,\mathcal{Z},\Omega)=\mathcal{G}(\R_{M}(\B),\R_{M}(\mathcal{Z}),\Omega)$.

\end{thm}
\begin{proof}
The proofs follow from the definitions.
\end{proof}
A consequence of Theorem \ref{thm:ograndi} and Proposition \ref{thm:acr}
is that the theory could be developed in terms of asymptotically closed
rings. This is the point of view followed by \foreignlanguage{british}{\cite{Del09,Has11}}.
Nevertheless, we think that it is useful to consider the notion of
asymptotic gauge because many growth conditions are more easily expressed
in terms of asymptotic gauges than in terms of asymptotically closed
rings: note e.g.\ that the assumption \eqref{eq:B-Z_O-totallyOrdered}
is too restrictive if $\B$ is a ring.
\begin{example}
\label{exa:fullAlgebra}Let $\mathcal{B}$ be an AG on $\mathbb{I}^{\srm}$,
then Thm. \ref{thm:AGfromInfinitesRho} yields that
\[
\B^{\erm}:=\left\{ \left(b_{\underline{\eps}}\right)_{\eps\in\mathcal{A}_{0}}\mid b\in\B\right\} \ ,\ \hat{\B}:=\left\{ \left(b_{\underline{\eps}}\right)_{\eps\in\mathcal{D}(\R^{d})}\mid b\in\B\right\} 
\]
are AG on $\mathbb{I}^{e}$ and $\hat{\mathbb{I}}$ respectively.
Therefore, if $\R_{M}(\B)\subseteq\R_{M}(\mathcal{Z})$ then $\mathcal{G}(\B^{\erm},\mathcal{Z}^{\erm},\Omega)$
and $\mathcal{G}(\hat{\B},\hat{\mathcal{Z}},\Omega)$ generalize the
full Colombeau algebra $\mathcal{G}^{\erm}(\Omega)$ and the algebra
$\hat{\mathcal{G}}(\Omega)$ of asymptotic functions (see also Cor.
\ref{cor:specialAlgebra} and Thm. \ref{thm:fullEquiv}).
\end{example}

\subsubsection{\label{sub:A-comparison}A comparison with asymptotic scales, $(\mathcal{C},\mathcal{E},\mathcal{P})$
algebras and exponent weights}

To study the relations between AG and the generalizations of Colombeau
algebras cited in the title, in this section we only consider the
set of indices $\mathbb{I}^{\srm}$ of the special algebra, the sheaf
$\Coo$of ordinary smooth functions, and the usual family of norms
\[
S:=\left\{ \|\partial^{\alpha}\cdot\|_{L^{\infty}(K)}\mid K\Subset\Omega,\alpha\in\N^{n}\right\} .
\]

Let $\B$, $\mathcal{Z}$ be AG, with $\R_{M}(\B)\subseteq\R_{M}(\mathcal{Z})$.
We already know (Thm. \ref{thm:acr}) that $\R_{M}(\B)$ is a solid
ring. Set
\begin{equation}
J_{\mathcal{Z}}:=\left\{ x\in\R^{I^{\srm}}\mid\forall z\in\mathcal{Z}_{>0}:\ x_{\eps}=O(z_{\eps}^{-1})\right\} .\label{eq:idealJ_Z}
\end{equation}
If $b\in\R_{M}(\B)$, $x\in J_{\mathcal{Z}}$ and $z\in\mathcal{Z}_{>0}$,
then Thm. \ref{thm:equivCondForIdeal}.\ref{enu:B-Z-relByRepresentatives}
yields $w_{\eps}^{-1}\cdot b_{\eps}=O(z_{\eps}^{-1})$ for some $w\in\mathcal{Z}_{>0}$.
By the definition \eqref{eq:idealJ_Z}, we have $x_{\eps}=O(w_{\eps}^{-1})$,
so that $x_{\eps}\cdot b_{\eps}=O(w_{\eps}^{-1}\cdot b_{\eps})=O(z_{\eps}^{-1})$,
which proves that $x\cdot b\in J_{\mathcal{Z}}$, i.e.\ that $J_{\mathcal{Z}}$
is an ideal of $\R_{M}(\B)$. Directly from the definitions, we get
\begin{equation}
\mathcal{A}_{\R_{M}(\B)/J_{\mathcal{Z}},{\scriptstyle \mathcal{C}^{\infty}},S}=\mathcal{G}(\B,\mathcal{Z},-),\label{eq:AG2CEP}
\end{equation}
where the left hand side is the $(\mathcal{C},\mathcal{E},\mathcal{P})$-algebra
associated to the ring $\R_{M}(\B)/J_{\mathcal{Z}}$.

Vice versa, let us assume that $A$ is a solid subring of $\R^{I^{\srm}}$
containing at least one infinite net
\[
\exists a\in A:\ \lim_{\eps\to0}a_{\eps}=\infty.
\]
Then $\B:=A$ is an AG. We can now consider the solid ideal associated
to $A$, i.e.
\[
I_{A}=\left\{ x\in\R^{I^{\srm}}\mid\forall a\in A^{*}:\ x=O(a)\right\} .
\]
It is easy to prove that $I_{A}=\left\{ x\in\R^{I^{\srm}}\mid\forall a\in A_{>0}^{*}:\ x=O(a^{-1})\right\} $
so that we can set $\mathcal{Z}:=I_{A}$, which is an AG. Directly
from the definitions, we get
\begin{equation}
\mathcal{A}_{A/I_{A},{\scriptstyle \mathcal{C}^{\infty}},S}=\mathcal{G}(\B,\mathcal{Z},-).\label{eq:CEP2AG}
\end{equation}
This proves that, in the framework of the special algebra, $(\mathcal{C},\mathcal{E},\mathcal{P})$-algebras
and Colombeau AG algebra are essentially equivalent.

In \cite{Has11}, the relations between $(\mathcal{C},\mathcal{E},\mathcal{P})$-algebras
and asymptotic algebras generated by an asymptotic scales are already
clarified, so that our previous \eqref{eq:AG2CEP}, \eqref{eq:CEP2AG}
would also give the relations with the latter. Anyway, let us assume
that $a=(a_{m})_{m\in\mathbb{Z}}$, $a_{m}\in\R_{>0}^{I^{\srm}}$,
is an asymptotic scale:
\[
\forall m\in\mathbb{Z}:\ a_{m+1}=o(a_{m}),a_{-m}=\frac{1}{a_{m}},\exists M\in\mathbb{Z}:\ a_{M}=O(a_{m}^{2}).
\]
Then, we can define the AG $\B:=\mathcal{Z}:=\left\{ x\in\R^{I^{\srm}}\mid\exists m\in\mathbb{Z}:\ x_{\eps}=O(a_{m\eps})\right\} $,
and this yields
\[
\mathcal{A}_{a}(\Coo,S)=\mathcal{G}(\B,\mathcal{Z},-).
\]
Finally, let $r:\N\ra\R_{>0}$ be a sequence of weights (see \cite{DPHV})
and, instead of the set of indices $\mathbb{I}^{\srm}$, we consider
the set of indices $\mathbb{I}^{\infty}:=(\N,\le,\mathcal{F})$, where
$\le$ is the usual order relation on $\N$ and $\mathcal{F}$ is
the Fr\'echet filter on $\N$. Setting
\[
\B:=\mathcal{Z}:=\{x\in\R^{\N}\mid\forall p\in S:\ \vvvert x\vvvert_{p,r}<\infty\}
\]
we get an AG and
\[
\mathcal{G}_{S,r}=\mathcal{G}(\B,\mathcal{Z},-).
\]
See \cite{DPHV} for more details and for the notations $\vvvert-\vvvert_{p,r}$
and $\mathcal{G}_{S,r}$.

\section{Embeddings of Distributions}

In this section we let $\Omega\subseteq\mathbb{R}^{n}$ be a fixed
open set and we let $\B,\mathcal{Z}$ be two fixed asymptotic gauges
with $\R_{M}(\B)\subseteq\R_{M}(\Z)$. We will define an embedding

\[
i:\mathcal{D}'(\Omega)\rightarrow\GBZ
\]

\noindent by slightly modifying the construction usually considered
in $\mathcal{G}^{s}(\Omega)$. Our construction will follow the same
approach used by \cite{GKOS}. We start by defining our mollificator.
\begin{defn}
\label{def:b_rho_rho_eps}Let $b\in\B_{>0}$ be infinite: $\lim_{\mathbb{I}}b=+\infty$;
for simplicity, we can assume that $b_{\eps}>0$ for all $\eps\in I$.
Let $\rho\in\mathcal{S}(\mathbb{R}^{n})$ be such that 
\begin{enumerate}[leftmargin=*,label=(\roman*),align=left ]
\item $\int\rho(x)\diff{x}=1$;
\item $\int\rho(x)x^{k}\diff{x}=0$ for every $k\geq1$. 
\end{enumerate}

\noindent We set 
\[
\rho_{\eps}:=b_{\eps}^{-1}\odot\rho,
\]
so that $\rho_{\eps}(x)=b_{\eps}^{n}\cdot\rho(b_{\eps}x)$ for all
$x\in\R^{n}$.

\end{defn}
\begin{lem}
\label{lem:convAndLimInD'}Let $w\in\mathcal{\mathcal{E}}'(\Omega)$
and $\phi\in\D(\Omega)$, then
\[
\lim_{\eps\in\mathbb{I}}\int_{\Omega}(w\ast\rho_{\eps})\cdot\phi=\langle w,\phi\rangle.
\]
\end{lem}
\begin{proof}
As usual, by the continuity of convolution, the conclusion is equivalent
to
\begin{equation}
\lim_{\eps\le e}\int\rho_{\eps}\cdot\phi=\phi(0),\label{eq:limitDelta}
\end{equation}
where $e\in I$. In fact, this would prove that $(\rho_{\eps})_{\eps\le e}\to\delta$
in $\D'$ with respect to the directed set $(\emptyset,e]$. To prove
\eqref{eq:limitDelta}, we consider
\begin{align*}
\left|\int\rho_{\eps}(x)\phi(x)\diff{x}-\phi(0)\right| & =\left|\int\rho(t)\left[\phi\left(\frac{t}{b_{\eps}}\right)-\phi(0)\right]\diff{t}\right|\le\\
 & \le\int\left|\rho(t)\right|\left|\phi\left(\frac{t}{b_{\eps}}\right)-\phi(0)\right|\diff{t}.
\end{align*}
Since $\phi$ is compactly supported, we have
\begin{align*}
\forall r\in\R_{>0}\,\exists\eps_{0}\le e\,\forall\eps\le\eps_{0}\, & \forall t\in\text{supp}\phi:\ \left|\phi\left(\frac{t}{b_{\eps}}\right)-\phi(0)\right|<r,
\end{align*}
 so 
\begin{align*}
\forall r\in\R_{>0}\,\exists\eps_{0}\le e\,\forall\eps\le\eps_{0}:\ \left|\int\rho_{\eps}(x)\phi(x)\diff{x}-\phi(0)\right|\le r\int\left|\rho\right|.
\end{align*}
\end{proof}
\begin{thm}
\label{thm:linearEmbedding}For every open $\Omega\subseteq\mathbb{R}^{n}$
the map 
\begin{align}
i_{0}:\mathcal{E}'(\Omega) & \rightarrow\GBZ\nonumber \\
w & \mapsto\left[(w\ast\rho_{\eps})|_{\Omega}\right]\label{eq:embeddingOnCmptSuppDistr}
\end{align}
is a linear embedding.\end{thm}
\begin{proof}
We have to prove that 
\begin{enumerate}[leftmargin=*,label=(\roman*),align=left ]
\item $i_{0}$ is linear;
\item \label{enu:imageEmb}$\forall w\in\mathcal{E}'(\Omega):\ (w\ast\rho_{\eps})|_{\Omega}\in\mathcal{E}_{M}(\B,\Omega)$;
\item \label{enu:kerInj}$\text{Ker}(i_{0})=\{0\}$. 
\end{enumerate}
That $i_{0}$ is linear follows immediately by the definition, since
the convolution is a linear operator. Let us prove \ref{enu:imageEmb}.
By the local structure theorem for distributions, it suffices to consider
the case $w=\partial^{\alpha}f\in\mathcal{E}'(\Omega)$, with $f\in\D(\Omega)$
and $\alpha\in\N^{n}$. Let $x\in K\Subset\Omega$, then 
\begin{align*}
(w\ast\rho_{\eps})(x) & =f\ast\partial^{\alpha}\rho_{\eps}(x)=\int f(x-y)\partial^{\alpha}\rho_{\epsilon}(y)\diff{y}=\\
 & =\int f(x-y)b_{\eps}^{n+|\alpha|}(\partial^{\alpha}\rho)(b_{\eps}\cdot y)\diff{y}=\\
 & =b_{\eps}^{|\alpha|}\int f\left(x-\frac{t}{b_{\eps}}\right)\cdot\partial^{\alpha}\rho(t)\diff{t}=O(b_{\eps}^{|\alpha|})=O(c_{\eps}),
\end{align*}
for some $(c_{\eps})\in\B$ with $b_{\eps}^{|\alpha|}=O(c_{\eps})$.
The same argument applies to the derivative $\partial^{\beta}(f\ast\partial^{\alpha}\rho_{\eps})=f\ast\partial^{\alpha+\beta}\rho_{\eps}$.

To prove \ref{enu:kerInj}, let $w\in\mathcal{E}'(\Omega)$ be such
that $\left[\left(w\ast\rho_{\eps}\right)|_{\Omega}\right]=0$ and
let $\phi\in\D(\Omega)$. Thus, setting $K:=\text{supp}(\phi)$, we
have
\[
\forall z\in\mathcal{Z}_{>0}:\ \sup_{x\in K}\left|(w\ast\rho_{\eps})(x)\right|=O(z_{\eps}^{-1})
\]
and hence $\lim_{\eps\in\mathbb{I}}\sup_{x\in K}\left|(w\ast\rho_{\eps})(x)\right|=0$.
From this and Lemma \ref{lem:convAndLimInD'}, we hence obtain
\[
\left|\langle w,\phi\rangle\right|=\lim_{\eps\in\mathbb{I}}\left|\int_{\Omega}(w\ast\rho_{\eps})\phi\right|\le\lim_{\eps\in\mathbb{I}}\sup_{x\in K}\left|(w\ast\rho_{\eps})(x)\right|\cdot\int_{K}\left|\phi\right|=0.
\]

\end{proof}
Let us note that the embedding \eqref{eq:embeddingOnCmptSuppDistr}
depends on the open set $\Omega$. We will use the notation $i_{0\Omega}$
when we want to underline this dependence.

We denote by $\sigma$ the constant embedding of $\mathcal{C}^{\infty}(\Omega)$
into $\GBZ$, namely $\sigma(f)=[f]$. We would like to prove the
analogue of \cite[Prop. 1.2.11]{GKOS}. As usual, the idea is to start
with $f\in\D(\Omega)$ and to use Taylor's formula obtaining 
\begin{align}
(f\ast\rho_{\eps}-f)(x)= & \int(f(x-y)-f(x))\rho_{\eps}(y)\diff{y}=\nonumber \\
= & \int\left(f\left(x-\frac{t}{b_{\eps}}\right)-f(x)\right)\rho(t)\diff{t}=\nonumber \\
= & \int\sum\limits _{0<|\alpha|<m}\frac{1}{\alpha!}\left(-\frac{t}{b_{\eps}}\right)^{\alpha}\partial^{\alpha}f(x)\rho(t)\diff{t}+\nonumber \\
 & +\int\sum_{|\alpha|=m}\frac{1}{\alpha!}\left(-\frac{t}{b_{\eps}}\right)^{\alpha}\partial^{\alpha}f\left(x-\Theta\frac{t}{b_{\eps}}\right)\rho(t)\diff{t}=\nonumber \\
=\  & 0+b_{\eps}^{-m}\cdot\int\sum_{|\alpha|=m}\frac{1}{\alpha!}\left(-t\right)^{\alpha}\partial^{\alpha}f\left(x-\Theta\frac{t}{b_{\eps}}\right)\rho(t)\diff{t}=\nonumber \\
=\  & O\left(b_{\eps}^{-m}\right).\label{eq:TaylorEmbedding}
\end{align}
Therefore, to have $(f-(f\ast\rho_{\eps})|_{\Omega})\in\mathcal{N}(\mathcal{Z},\Omega)$
we need a further condition of the form
\[
\forall z\in\mathcal{Z}_{>0}\,\exists m\in\N:\ b_{\eps}^{-m}=O(z_{\eps}^{-1})\quad\text{i.e.}\quad z_{\eps}=O(b_{\eps}^{m}).
\]
This implies $\mathcal{Z}_{>0}\subseteq\R_{M}(\text{AG}(b))\subseteq\R_{M}(\B)$,
where $\text{AG}(b)$ is the AG 
\[
\text{AG}(b):=\{b^{m}\mid m\in\N\}.
\]
We have thus $\R_{M}(\mathcal{Z})=\R_{M}(\B)=\R_{M}(\text{AG}(b))$.
\begin{defn}
Let $\B$ be an AG. If $b\in\B$ is such that $\R_{M}(\text{AG}(b))\supseteq\R_{M}(\B)$
then we will say that $b$ is a \emph{generator of $\B$} and that
$\B$ is a \emph{principal AG.}
\end{defn}
Let us note that in the previous definition we could equivalently
substitute the condition $\R_{M}(\text{AG}(b))\supseteq\R_{M}(\B)$
with $\R_{M}(\text{AG}(b))=\R_{M}(\B)$. Moreover if $\B$ is principal
AG then, if necessary, we can always find a positive generator of
$\B$: in fact, if $b$ is any generator of \emph{$\B$} and $c\in\B_{>0},\ c\in O(|b|)$,
then also $c$ is a generator of \emph{$\B$}. 

Every AG of Ex. \ref{exa:AG}, other than $\B_{\text{fin}}^{\text{exp}}$
and $\hat{\B}_{\text{fin}}^{\text{exp}}$, is principal; for example,
$\eps^{-1}$ is a generator of $\mathcal{B}^{\srm}$. Moreover, a
solid subalgebra of $\R^{I}$ (containing an infinite net) generally
speaking is not a principal AG. In the latter case, the embedding
of distributions using a mollifier is not possible (see Thm. \ref{thm:principalNecessary}).

By Theorem \ref{thm:ograndi} we can also assume, without loss of
generality, that $\B=\mathcal{Z}$, which is the subject of our next
\begin{description}
\item [{Assumption}] \noindent $\mathcal{Z}=\B$ is a principal AG. Moreover
we assume that the mollifier $\left(\re\right)$ is constructed with
a fixed generator $b$.
\end{description}
Let us observe that a generator of an asymptotic gauge is necessarily
an infinite element in $\mathbb{R}^{I}$, i.e.\ $\lim_{\eps\in\mathbb{I}}b_{\eps}=\infty$,
and that every principal AG is totally ordered. As a first consequence
of our assumption, we have 
\begin{multline*}
\mathcal{N}(\mathcal{Z},\Omega)=\mathcal{N}(\B,\Omega)=\\
\left\{ u\in\mathcal{C}^{\infty}(\Omega)^{I}\mid\forall K\Subset\Omega\,\forall\alpha\in\mathbb{N}^{n}\,\forall m\in\mathbb{N}:\ \sup\limits _{x\in K}|\partial^{\alpha}u_{\eps}(x)|=O(b_{\eps}^{-m})\text{ as }\eps\in\mathbb{I}\right\} .
\end{multline*}
Henceforward, we will thus use the simplified notation $\mathcal{G}(\B,\Omega):=\mathcal{G}(\B,\B,\Omega)$.
\begin{thm}
$i_{0}|_{\mathcal{D}(\Omega)}=\sigma$. Consequently, $i_{0}$ is
an injective homomorphism of algebras on $\mathcal{D}(\Omega)$.\end{thm}
\begin{proof}
The second statement follows from the first like in \cite{GKOS}.
The remaining part is proved in \eqref{eq:TaylorEmbedding}.
\end{proof}
\noindent The notions of support $\text{supp}(u)$ of a generalized
function $u\in\mathcal{G}(\B,\Omega)$ and of restriction $u|_{\Omega'}\in\mathcal{G}(\B,\Omega')$
can be defined exactly like in \cite[pag. 12]{GKOS}.
\begin{thm}
If $w\in\mathcal{E}'(\Omega)$ then $\text{\text{\emph{supp}}}(w)=\text{\emph{\text{supp}}}(i_{0}(w))$.\end{thm}
\begin{proof}
Let us prove that $\text{\text{supp}}(i_{0}(w))\subseteq\text{\text{supp}}(w)$.
We have to prove that 
\[
i_{0}(w)|_{\text{\text{supp}}(w)^{c}}=0
\]
in $\mathcal{G}(\B,\text{\text{supp}}(w)^{c})$. Let $K\Subset\text{\text{supp}}(w)^{c}$,
let $\alpha\in\mathbb{N}^{n}$ be such that $w=\partial^{\alpha}f$,
with $f\in\D(\mathbb{R}^{n}\setminus K)$. Then $i_{0}(w)=\left[(f\ast\partial^{\alpha}\rho_{\eps})|_{\Omega}\right]$
and 
\begin{align*}
(f\ast\partial^{\alpha}\rho_{\eps})(x) & =\int f(x-y)\partial^{\alpha}\rho_{\eps}(y)\diff{y}=\\
 & =\int f(x-y)\cdot b_{\eps}^{|\alpha|+n}\partial^{\alpha}\rho(b_{\eps}y)\diff{y}=\\
 & =\int f\left(x-\frac{t}{b_{\eps}}\right)\cdot b_{\eps}^{|\alpha|}\partial\rho(t)\diff{t}=\\
 & =\int_{|t|<\sqrt{b_{\eps}}}f\left(x-\frac{t}{b_{\eps}}\right)\cdot b_{\eps}^{|\alpha|}\partial\rho(t)\diff{t}+\\
 & +\int_{|t|\geq\sqrt{b_{\eps}}}f\left(x-\frac{t}{b_{\eps}}\right)\cdot b_{\eps}^{|\alpha|}\partial\rho(t)\diff{t}.
\end{align*}
Recall that $b_{\eps}>0$ for each $\eps\in I$. Since $\text{\text{supp}}(f)\cap K=\emptyset$,
if $x\in K$ there exist $A\in\mathcal{I}$ such that for each $a\in A$
\[
\forall^{\mathbb{I}}\eps\in A_{\le a}\,\forall t:\ |t|<\sqrt{b_{\eps}}\then x-\frac{t}{b_{\eps}}\notin\text{supp}(f),
\]
so the first integral is zero. For the same $\eps$ sufficiently small,
the second integral can be estimated as follows: 
\[
\int_{|t|\geq\sqrt{b_{\eps}}}f\left(x-\frac{t}{b_{\eps}}\right)\cdot b_{\eps}^{|\alpha|}\partial\rho(t)\diff{t}\leq b_{\eps}^{|\alpha|}\cdot\left\Vert f\right\Vert _{\infty}\cdot\int_{|t|\geq\sqrt{b_{\eps}}}\left|\partial^{|\alpha|}\rho(t)\right|\diff{t}.
\]
Since $\rho\in\mathcal{S}(\mathbb{R}^{n})$ for any $m\in\N$ there
exists a constant $c_{m}>0$ such that $\left|\partial^{|\alpha|}\rho(t)\right|\leq c_{m}(1+|t|)^{-2m-n-1}$.
Thus $\left|\partial^{|\alpha|}\rho(t)\right|\leq c_{m}\left(\sqrt{b_{\eps}}\right)^{-2m}(1+|t|)^{-n-1}$
for $|t|\ge\sqrt{b_{\eps}}$, and 
\[
\int_{|t|\geq\sqrt{b_{\eps}}}b_{\eps}^{|\alpha|}\left|\partial^{|\alpha|}\rho(t)\right|\diff{t}\leq c_{m}b_{\eps}^{|\alpha|-m}\cdot\int_{|t|\geq\sqrt{b_{\eps}}}(1+|t|)^{-n-1}\diff{t}=\widetilde{c_{m}}\cdot b_{\eps}^{|\alpha|-m}.
\]
Since $m$ is arbitrary we can treat the derivative of $i_{0}(w)$
in the same way, and this gives the desired estimates that show that
$\text{\text{\text{supp}}}(i_{0}(w))\subseteq\text{\text{\text{supp}}}(w)$.

Let us now prove that $\text{\text{supp}}(w)\subseteq\text{\text{supp}}(i_{0}(w))$.
Let $x_{0}\in\text{\text{supp}}(w)$. For every $\eta>0$ there exists
$\varphi\neq0$ in $\mathcal{D}(\mathbb{R})$ such that $\text{\text{supp}}(\varphi)\subseteq B_{\eta}(x_{0})$
and $|\langle w,\varphi\rangle|=:c>0$. Since $\lim_{\eps\in\mathbb{I}}w\ast\rho_{\eps}=w$
in $\D'(\Omega)$ (Lemma \ref{lem:convAndLimInD'}), this implies
that $|\langle w\ast\rho_{\eps},\varphi\rangle|>\frac{c}{2}$ for
$\eps\in\mathbb{I}$ small. But setting $K:=\overline{B_{\eta}(x_{0})}$
we have
\begin{equation}
0<\frac{c}{2}<\left|\langle w\ast\rho_{\eps},\phi\rangle\right|=\left|\int_{K}(w\ast\rho_{\eps})(x)\cdot\phi(x)\diff{x}\right|\le\sup_{x\in K}\left|(w\ast\rho_{\eps})(x)\right|\cdot\int_{K}\left|\phi\right|.\label{eq:ineq2ndInclusionSupp}
\end{equation}
The equality $\left[(w\ast\rho_{\eps})|_{B_{\eta}(x_{0})}\right]=0$
in $\mathcal{G}(\B,B_{\eta}(x_{0}))$ would imply
\[
\lim_{\eps\in\mathbb{I}}\sup_{x\in K}\left|(w\ast\rho_{\eps})(x)\right|=0,
\]
which is impossible by \eqref{eq:ineq2ndInclusionSupp}, so $i_{0}(w)|_{B_{\eta}(x_{0})}=\left[(w\ast\rho_{\eps})|_{B_{\eta}(x_{0})}\right]\neq0$
and therefore $x_{0}\in\text{\text{supp}}(i_{0}(w))$.
\end{proof}
To prove that $i_{0}$ can be extended to an embedding $i:\mathcal{D}'(\Omega)\rightarrow\mathcal{G}$
we can now use the following result, whose proof can be conducted
exactly like in \cite{GKOS}:
\begin{thm}
\ 
\begin{enumerate}[leftmargin=*,label=(\roman*),align=left ]
\item $\mathcal{G}(\B,-):\Omega\mapsto\mathcal{G}(\B,\Omega)$ is a sheaf
of differential algebras on $\R^{n}$;
\item There is a unique sheaf morphism of real vector spaces $i:\D'\ra\mathcal{G}(\B,-)$
such that:

\begin{enumerate}
\item $i$ extends the embedding $i_{0}:\mathcal{E}'\ra\mathcal{G}(\B,-)$
defined in \eqref{eq:embeddingOnCmptSuppDistr}, i.e.\ such that
$i_{\Omega}|_{\mathcal{E}'(\Omega)}=i_{0\Omega}$.
\item $i$ commutes with partial derivatives, i.e.\ $\partial^{\alpha}\left(i_{\Omega}(T)\right)=i_{\Omega}\left(\partial^{\alpha}T\right)$
for each $T\in\D'(\Omega)$ and $\alpha\in\N$.
\item $i|_{\Coo(-)}:\Coo(-)\ra\mathcal{G}(\B,-)$ is a sheaf morphism of
algebras.
\end{enumerate}
\end{enumerate}
\end{thm}

\subsection{Embedding with a strict $\delta$-net}

A simpler way to embed distributions is by means of a strict $\delta$-net
rather than a model $\delta$-net (i.e.\ a net obtained by scaling
a single function $\rho$). This can be done for a generic set of
indices and a principal AG simply by generalizing \cite[Lem. A1, Cor. A2]{SteVic09}:
\begin{thm}
\label{thm:strictDeltaNet}Let $\mathbb{I}$ be a set of indices and
$\B$ be a principal AG on $\mathbb{I}$ generated by $b\in\B_{>0}$.
There exists a net $\left(\psi_{\eps}\right)_{\eps\in I}$ of $\D(\R^{n})$
with the properties:
\begin{enumerate}[leftmargin=*,label=(\roman*),align=left ]
\item \label{enu:suppStrictDeltaNet}$supp(\psi_{\eps})\subseteq B_{1}(0)$
for all $\eps\in I$;
\item \label{enu:intOneStrictDeltaNet}$\int\psi_{\eps}=1$ for all $\eps\in I$;
\item \label{enu:moderateStrictDeltaNet}$\forall\alpha\in\N^{n}\,\exists p\in\N:\ \sup_{x\in\R^{n}}\left|\partial^{\alpha}\psi_{\eps}(x)\right|=O(b_{\eps}^{p})$
as $\eps\in\mathbb{I};$
\item \label{enu:momentsStrictDeltaNet}$\forall j\in\N\,\forall^{\mathbb{I}}\eps:\ 1\le|\alpha|\le j\Rightarrow\int x^{\alpha}\cdot\psi_{\eps}(x)\diff{x}=0;$
\item \label{enu:smallNegPartStrictDeltaNet}$\forall\eta\in\R_{>0}\,\forall^{\mathbb{I}}\eps:\ \int\left|\psi_{\eps}\right|\le1+\eta$.
\end{enumerate}

\noindent In particular
\[
\rho_{\eps}:=b_{\eps}^{-1}\odot\psi_{\eps}\quad\forall\eps\in I
\]
satisfies \ref{enu:intOneStrictDeltaNet} - \ref{enu:smallNegPartStrictDeltaNet}.

\end{thm}
\begin{proof}
For $m\in\N$ and $\eta\in\R_{>0}$ define the sets
\[
\mathcal{A}_{m}:=\left\{ \phi\in\D(\R^{n})\mid\text{supp}(\phi)\subseteq B_{1}(0),\int\phi=1,\int x^{\alpha}\phi(x)\diff{x}=0\ 1\le|\alpha|\le m\right\} ,
\]
\[
\mathcal{A}'_{m}(\eta):=\left\{ \phi\in\mathcal{A}_{m}\mid\int|\phi|\le1+\eta\right\} .
\]
In \cite{SteVic09} it is proved that $\mathcal{A}_{m}\ne\emptyset\ne\mathcal{A}'_{m}(\eta)$.
For each $m\in\N_{>0}$, we choose $\phi_{m}\in\mathcal{A}'_{m}\left(\frac{1}{m}\right)$
and we set
\[
M_{m}:=\sup_{\substack{x\in\R^{n}\\
|\alpha|\le m
}
}\left|\partial^{\alpha}\phi_{m}(x)\right|,
\]
\[
\mathcal{A}_{m,\eps}:=\left\{ \phi\in\mathcal{A}_{m}'\left(\frac{1}{m}\right)\mid\sup_{\substack{x\in\R^{n}\\
|\alpha|\le m
}
}\left|\partial^{\alpha}\phi(x)\right|\le b_{\eps}\right\} \quad\forall m\in\N_{>0}\,\forall\eps\in I.
\]
Therefore, $\emptyset\ne\mathcal{A}_{m+1,\eps}\subseteq\mathcal{A}_{m,\eps}$
and $\phi_{m}\in\mathcal{A}_{m,\eps}$ whenever $M_{m}\le b_{\eps}$.
Since $\lim_{m\to+\infty}M_{m}=+\infty$ (see \cite{SteVic09}), for
each fixed $\eps\in I$, we have $b_{\eps}<M_{m+1}$ for $m$ sufficiently
big. We denote by $m_{\eps}$ the minimum $m\in\N$ such that $b_{\eps}<M_{m+1}$,
so that
\begin{equation}
M_{m_{\eps}}\le b_{\eps}<M_{m_{\eps}+1}\quad\forall\eps\in I\label{eq:Def-m_eps}
\end{equation}
and hence $\phi_{m_{\eps}}\in\mathcal{A}_{m_{\eps},\eps}$ for all
$\eps\in I$. Define $\psi_{\eps}:=\phi_{m_{\eps}}$ for all $\eps\in I$,
so that $\psi_{\eps}=\phi_{m_{\eps}}\in\mathcal{A}_{m_{\eps},\eps}\subseteq\mathcal{A}_{m_{\eps}}$,
which proves \ref{enu:suppStrictDeltaNet}, \ref{enu:intOneStrictDeltaNet}.
The remaining properties can be proved like in \cite{SteVic09}. We
have only to note that if $\alpha\in\N^{n}$, then $|\alpha|\le b_{\eps}$
for $\eps\in\mathbb{I}$ sufficiently small because $\lim_{\eps\in\mathbb{I}}b_{\eps}=+\infty$.
Therefore, \eqref{eq:Def-m_eps} yields $|\alpha|\le m_{\eps}$ and
hence $\mathcal{A}_{|\alpha|,\eps}\supseteq\mathcal{A}_{m_{\eps},\eps}\ni\psi_{\eps}$.
\end{proof}
We finally have the following results, whose proof is just a mild
variation of the previous proofs about the embedding with a model
$\delta$-net.
\begin{cor}
If $(\rho_{\eps})$ is the net defined like in Thm. \ref{thm:strictDeltaNet},
then the mapping
\[
i_{\Omega}:T\in\D'(\Omega)\mapsto\left[T\ast\rho_{\eps}\right]\in\mathcal{G}(\B,\Omega)
\]
is a sheaf morphism of real vector spaces $i:\D'\ra\mathcal{G}(\B,-)$,
and satisfies the following properties:
\begin{enumerate}[leftmargin=*,label=(\roman*),align=left ]
\item $i$ commutes with partial derivatives, i.e.\ $\partial^{\alpha}\left(i_{\Omega}(T)\right)=i_{\Omega}\left(\partial^{\alpha}T\right)$
for each $T\in\D'(\Omega)$ and $\alpha\in\N;$
\item $i|_{\Coo(-)}:\Coo(-)\ra\mathcal{G}(\B,-)$ is a sheaf morphism of
algebras;
\item If $w\in\mathcal{E}'(\Omega)$ then $\text{\text{\emph{supp}}}(w)=\text{\emph{\text{supp}}}(i_{\Omega}(w))$;
\item $i_{\Omega}(T)\approx T$ for each $T\in\D'(\Omega)$, i.e.\ $\lim_{\eps\in\mathbb{I}}\int_{\Omega}\left(T\ast\rho_{\eps}\right)\cdot\phi=\langle T,\phi\rangle$
for all $\phi\in\D(\Omega)$.
\end{enumerate}
\end{cor}

\subsection{Comparison of embeddings}

We close this section by facing a natural problem: let us define two
embeddings $i_{b}$, $i_{c}$ like \eqref{eq:embeddingOnCmptSuppDistr}
but using two different generators $b$, $c\in\B$:
\begin{align*}
i_{b}(w): & =\left[w\ast(b_{\eps}^{-1}\odot\rho)\right],\\
i_{c}(w) & =\left[w\ast(c_{\eps}^{-1}\odot\rho)\right].
\end{align*}
It is well known that $i_{b}(T)\approx T\approx i_{c}(T)$, but when
are they equal?
\begin{thm}
Let $b$, $c\in\B_{>0}$ be generators of the AG $\B$. Assume that
$\rho(0)\ne0$; then $i_{b}=i_{c}$ if and only if $[b_{\eps}]=[c_{\eps}]$
in $\mathcal{G}(\B,\R)$, i.e.\ iff they generate the same $\B$-Colombeau
generalized number.\end{thm}
\begin{proof}
If $i_{b}=i_{c}$, then $i_{b}(\delta)=\left[b_{\eps}^{n}\cdot\rho(b_{\eps}\cdot-)\right]=i_{c}(\delta)=\left[c_{\eps}^{n}\cdot\rho(c_{\eps}\cdot-)\right]$.
Setting $K=\{0\}$ in the definition of negligible net, we get
\[
\forall m\in\N:\ \left|b_{\eps}^{n}\rho(0)-c_{\eps}^{n}\rho(0)\right|=O(b_{\eps}^{-m}),
\]
that is $[b_{\eps}^{n}]=[c_{\eps}^{n}]$. The conclusion follows by
applying the smooth function $\sqrt[n]{-}\in\Coo(\R_{>0})$.

Vice versa, assume that $[b_{\eps}]=[c_{\eps}]$; we want to prove
that
\[
\left[w\ast\left(b_{\eps}^{-1}\odot\rho-c_{\eps}^{-1}\odot\rho\right)\right]=0\quad\forall w\in\mathcal{E}'(\Omega).
\]
It suffices to prove that $\lim_{\eps\in\mathbb{I}}\left(b_{\eps}^{-1}\odot\rho-c_{\eps}^{-1}\odot\rho\right)=0$
in $\D'(\Omega)$. For each $\phi\in\D(\Omega)$ we have
\begin{equation}
\int\left(b_{\eps}^{-1}\odot\rho-c_{\eps}^{-1}\odot\rho\right)\phi=\int\rho(t)\cdot\left[\phi\left(\frac{t}{b_{\eps}}\right)-\phi\left(\frac{t}{c_{\eps}}\right)\right].\label{eq:zeroDistrTwoEmbeddings}
\end{equation}
The composition of the generalized functions $\phi\in\Coo(\R^{n})\subseteq\mathcal{G}(\B,\R^{n})$
and
\[
\left[x\mapsto\frac{x}{b_{\eps}}\right]=\left[x\mapsto\frac{x}{c_{\eps}}\right]\in\mathcal{G}(\B,\R^{n})
\]
is well defined since the latter is compactly supported. Therefore,
for $K:=\text{supp}(\phi)$, $\lim_{\eps\in\mathbb{I}}\sup_{t\in K}\left|\phi\left(\frac{t}{b_{\eps}}\right)-\phi\left(\frac{t}{c_{\eps}}\right)\right|=0$.
From this and \eqref{eq:zeroDistrTwoEmbeddings} the conclusion follows.
\end{proof}
The assumption $\rho(0)\ne0$ clearly holds if we define $\rho$ as
the inverse Fourier transform of a positive function identically equal
to 1 in a neighborhood of 0.

For example, $i_{(\eps^{-k})}$ and $i_{(\eps^{-h})}$ permit to deal
with different speeds at the origin of different models of the Heaviside
function $H$. Finally, as we already said at the beginning of the
present work, it could be interesting to apply these results about
different embeddings also to the full algebra $\mathcal{G}^{e}(\B^{e},\Omega)$,
e.g.\ in case we need particular properties like $H(0)=0$. This
is only a first step in the study of the infinitesimal (and infinite)
differences between two embeddings $i_{b}$ and $i_{c}$. In our opinion,
this study could be very useful in nonlinear modeling.

\subsection{Necessity of a principal AG to embed distributions with a mollifier}

The assumption that $\B$ is a principal AG is quite natural if one
looks at \eqref{eq:TaylorEmbedding}. In this section we want to prove
that this is indeed a necessary condition if we want to have a pair
$i_{0}$, $\sigma$ of embeddings (where $i_{0}$ is defined like
in Thm. \ref{thm:linearEmbedding}) which coincide on a suitable set.
More precisely, to state the following result, we set
\begin{defn}
Let $\B$ be an AG, then $\mathcal{E}'_{M}(\B,\Omega):=\mathcal{E}_{M}(\B,\Omega)\cap\D(\Omega)^{I}$
denotes the set of moderate nets of compactly supported functions.
\end{defn}
\noindent We also recall that if $(z_{k})_{k}$ is a sequence of $A_{\le a}$,
then we say $(z_{k})_{k}\to\emptyset$ in $A_{\le a}$ if 
\[
\forall a_{0}\in A_{\le a}\,\exists K\in\N\,\forall k\in\N_{\ge K}:\ z_{k}<a_{0}.
\]
The existence of such a sequence is always verified in all our examples
of set of indices (see \cite{GiNi14}).
\begin{thm}
\label{thm:principalNecessary}Let $\B$, $\mathcal{Z}$ be AG on
the set of indices $\mathbb{I}$. Assume that for each $a\in A\in\mathcal{I}$
there exists a sequence $(z_{k})_{k}\to\emptyset$ in $A_{\le a}$.
Let $b$, $\rho$, $\rho_{\eps}$ as in Def. \ref{def:b_rho_rho_eps}.
Then the following are equivalent:
\begin{enumerate}[leftmargin=*,label=(\roman*),align=left ]
\item \label{enu:coherentEmbeddings}$\forall(f_{\eps})\in\mathcal{E}'_{M}(\text{\emph{AG}}(b),\R)\,\forall x\in\R:\ \left(f_{\eps}\ast\rho_{\eps}\right)(x)=f_{\eps}(x)+\mathcal{N}(\mbox{\ensuremath{\mathcal{Z}}},\R);$
\item \label{enu:b-Generator}$b$ is a generator of $\mathcal{Z}$.
\end{enumerate}
\end{thm}
\begin{proof}
We prove \ref{enu:coherentEmbeddings} $\Rightarrow$ \ref{enu:b-Generator}
only for the case $n=1$, even if slightly more general notations
can be used to repeat this proof for a generic dimension. As in \eqref{eq:TaylorEmbedding},
we can use in \ref{enu:coherentEmbeddings} a Taylor formula of order
$m\in\N$, with Peano remainder, at $x\in\Omega$, so that for each
$f_{\eps}\in\D(\Omega)$ we have a (unique) remainder $R_{\eps}=R(m,f_{\eps},x)\in\D(\Omega)$
such that for each $\eps\in I$ 
\begin{align}
\left|\left(f_{\eps}\ast\rho_{\eps}\right)(x)-f_{\eps}(x)\right| & =b_{\eps}^{-m}\cdot\int_{\R}R_{\eps}\left(x-\frac{t}{b_{\eps}}\right)\cdot\rho(t)\diff{t}\nonumber \\
R_{\eps}(y) & =o\left(\left|y-x\right|^{m}\right)\text{ as }y\to x.\label{eq:Peano}
\end{align}
We set $c_{\eps}(m,f_{\eps},x):=\int_{\R}R_{\eps}\left(x-\frac{t}{b_{\eps}}\right)\cdot\rho(t)\diff{t}$.
Without lack of generality, we can assume that $\rho(0)>0$; analogously,
we can proceed if $\rho(x)\ne0$ at another point $x\in\R$.

Now, for each $\eps\in I$ and $m\in\N_{>0}$, we want to define a
function $f_{\eps m}$ to use in \eqref{eq:Peano} such that:
\begin{itemize}
\item $f_{\eps m}$ is equal to its Peano remainder of order $m$ at $x=0$.
This permits to directly have $R_{\eps}=f_{\eps m}$ in \eqref{eq:Peano}
and in the definition of $c_{\eps}(m,f_{\eps m},0)$.
\item $c_{\eps}(m,f_{\eps m},0)\ge L_{m}>0$, where $L_{m}$ doesn't depend
on $\eps$ and is infinitesimal for $m\to+\infty$.
\end{itemize}

\noindent $f_{\eps m}$ can be defined in infinite ways; in its definition
we will always respect the following criteria:
\begin{enumerate}[leftmargin=*,label=(\roman*),align=left ]
\item \label{enu:pq}We firstly fix $p$, $q\in\R_{>0}$ such that
\begin{align*}
\rho(t) & >0\quad\forall t\in(-p,p)\\
q & <\min(p,1).
\end{align*}

\item \label{enu:supp_f_eps_m}$f_{\eps m}\in\mathcal{E}_{M}(\text{AG}(b),\R)$
and $\text{supp}(f_{\eps m})\subseteq\left[-\frac{p}{b_{\eps}},\frac{p}{b_{\eps}}\right]$.
\item \label{enu:f_eps_m-positive}$f_{\eps m}(s)\ge0$ for each $s$.
\item \label{enu:f_eps_mEqualPeanoRem}$f_{\eps m}(s)=s^{2m}\cdot b_{\eps}^{2m}$
for each $s\in\left(-\frac{q}{b_{\eps}},\frac{q}{b_{\eps}}\right).$
\end{enumerate}

For example, we can take $\psi\in\D(\R)$ such that $\text{supp}(\psi)\subseteq[-p,p]$,
$\psi\ge0$, $\psi(s)=1$ for $s\in[-q,q]$ and set $f_{\eps m}(s)=s^{2m}\cdot b_{\eps}^{2m}\cdot\psi\left(b_{\eps}\cdot s\right)$;
let us note that $\lim_{\eps\in\mathbb{I}}\left[f_{\eps m}\right]=s^{2m}\cdot b^{2m-n}\cdot\delta(s)$
in $\mathcal{G}(\text{AG}(b),\R)$ in the sharp topology, so that
the limit of this net is a nonlinear generalized Colombeau function.

\noindent We have $f_{\eps m}(s)=o(s^{m+1})$ as $s\to0$ by \ref{enu:f_eps_mEqualPeanoRem},
so $f_{\eps m}$ equals its Taylor remainder of order $m>0$. Moreover
$t\in(-p,p)$ iff $-\frac{t}{b_{\eps}}\in\left(-\frac{p}{b_{\eps}},\frac{p}{b_{\eps}}\right)$,
so that
\begin{equation}
c_{\eps}(m,f_{\eps m},0)=\int_{-p}^{p}f_{\eps m}\left(-\frac{t}{b_{\eps}}\right)\cdot\rho(t)\diff{t}\ge\int_{-q}^{q}t^{2m}\cdot\rho(t)\diff{t}=:L_{m}>0,\label{eq:c_eps_L_m}
\end{equation}
where we have used \ref{enu:f_eps_mEqualPeanoRem}, \ref{enu:f_eps_m-positive}
and \ref{enu:supp_f_eps_m}. Since $q<1$, $t^{2m}\cdot\rho(t)<\rho(t)$
for every $t\in[-q,q].$ So, by dominated convergence, $\lim_{m\to+\infty}L_{m}=0$.
Now, \eqref{eq:Peano} yields 
\[
\left|\left(f_{\eps m}\ast\rho_{\eps}\right)(0)-f_{\eps m}(0)\right|=b_{\eps}^{-m}\cdot c_{\eps}(m,f_{\eps m},0)\quad\forall\eps,m,
\]
so that, considering a generic $z\in\mathcal{Z}_{>0}$, assumption
\ref{enu:coherentEmbeddings} gives
\begin{equation}
\forall m\in\N_{>0}\,\exists A_{m}\in\mathcal{I}\,\forall a\in A_{m}:\ b_{\eps}^{-m}\cdot c_{\eps}(m,f_{\eps m},0)=O_{a,A_{m}}(z_{\eps}^{-1}).\label{eq:c_epsO}
\end{equation}
We proceed by contradiction assuming that
\[
\forall m\in\N:\ b_{\eps}^{-m}\ne O(z_{\eps}^{-1}).
\]
Taking a generic $m\in\N_{>0}$, this means
\[
\forall A\in\mathcal{I}\,\exists a\in A:\ b_{\eps}^{-m}\ne O_{a,A}(z_{\eps}^{-1}).
\]
We apply this with $A=A_{m}$ obtaining
\[
\exists a_{m}\in A_{m}:\ b_{\eps}^{-m}\ne O_{a_{m},A_{m}}(z_{\eps}^{-1}).
\]
By Thm. 15 of \cite{GiNi14}, we obtain that for each $H\in\R_{>0}$
there exists a sequence $(\eps_{k})_{k\in\N}:\N\ra\left(A_{m}\right)_{\le a_{m}}$
(depending on $m$ and $H$) such that: 
\begin{align}
(\eps_{k})_{k} & \to\emptyset\text{ in }\left(A_{m}\right)_{\le a_{m}}\nonumber \\
b_{\eps_{k}}^{-m} & >H\cdot z_{\eps_{k}}^{-1}\quad\forall k\in\N.\label{eq:H}
\end{align}
We set $H:=L_{m}^{-2}>0$ in \eqref{eq:H}, obtaining a sequence $(\eps_{k})_{k}\to\emptyset\text{ in }\left(A_{m}\right)_{\le a_{m}}$
(depending only on $m$) such that
\begin{equation}
b_{\eps_{k}}^{-m}z_{\eps_{k}}>L_{m}^{-2}\quad\forall k,m\in\N.\label{eq:bzL}
\end{equation}
But from \eqref{eq:c_epsO} with $a=a_{m}$ and \eqref{eq:c_eps_L_m},
we get
\[
\forall m\in\N_{>0}\,\exists T\in\R_{>0}\,\forall^{\mathbb{I}}\eps\in\left(A_{m}\right)_{\le a_{m}}:\ b_{\eps}^{-m}\cdot z_{\eps}\le T\cdot c_{\eps}(m,f_{\eps m},0)^{-1}\le T\cdot L_{m}^{-1}.
\]
Applying this for $\eps=\eps_{k}$, with $k$ sufficiently big, we
get $b_{\eps_{k}}^{-m}z_{\eps_{k}}\le T\cdot L_{m}^{-1}$ which, together
with \eqref{eq:bzL}, yields $L_{m}^{-2}<T\cdot L_{m}^{-1}$ for each
$m$. This is impossible since $L_{m}\to0$ for $m\to+\infty$.

To prove \ref{enu:b-Generator} $\Rightarrow$ \ref{enu:coherentEmbeddings},
we can proceed as in \eqref{eq:TaylorEmbedding} considering that
$\partial^{\alpha}f_{\eps}$ is bounded, on a fixed compact sets $K\Subset\R$,
by a suitable power $b_{\eps}^{-N_{\alpha}}$. Therefore, for $x\in\R$
fixed and $\eps$ sufficiently small, $x-\frac{t}{b_{\eps}}\in\overline{B_{1}(x)}=:K$,
and for each $m\in\N$, we can write
\begin{align*}
\left|(f_{\eps}\ast\rho_{\eps})(x)-f_{\eps}(x)\right| & \le b_{\eps}^{-m}\cdot\sum_{|\alpha|=m}\sup_{x\in K}\left|\partial^{\alpha}f_{\eps}(x)\right|\int\frac{\left|t\right|{}^{\alpha}}{\alpha!}\left|\rho(t)\right|\diff{t}\le\\
 & \le b_{\eps}^{-m}\sum_{|\alpha|=m}b_{\eps}^{-N_{\alpha}}\int\frac{\left|t\right|{}^{\alpha}}{\alpha!}\left|\rho(t)\right|\diff{t}\le b_{\eps}^{-m-N_{m}}\cdot C_{m},
\end{align*}
where $N_{m}:=\max_{|\alpha|=m}N_{\alpha}$ and $C_{m}:=\sum_{|\alpha|=m}\int\frac{\left|t\right|{}^{\alpha}}{\alpha!}\left|\rho(t)\right|\diff{t}$.
If $z\in\mathcal{Z}_{>0}$, by \ref{enu:b-Generator} we get the existence
of $k\in\N$ such that $b_{\eps}^{-k}=O(z_{\eps}^{-1})$. It suffices
to take $m$ sufficiently big so that $m+N_{m}>k$ so that for this
fixed $m$ and for $\eps$ small, $b_{\eps}^{-m-N_{m}}\cdot C_{m}\le b_{\eps}^{-k}$.
\end{proof}
\noindent Let us note that condition \ref{enu:coherentEmbeddings}
is stronger than the equality $i_{0}|_{\mathcal{D}(\Omega)}=\sigma$,
which can be applied only to a single function $f\in\D(\Omega)$ instead
of a whole net. Indeed, in the previous proof, we used this condition
with the net $(f_{\eps m})_{\eps}$, which effectively depend on $\eps$.

\noindent We can say that if $b$ is a generator of $\mathcal{Z}$,
then the equality $i_{0}(f)=\sigma(f)$ for $f\in\D(\Omega)$ can
be extended to any net of compactly supported function which are $\text{AG}(b)$-moderated.
Therefore, the only possibility to have an embedding using a mollifier
but without using a generator is to avoid a natural property like
\ref{enu:coherentEmbeddings}, which is undesirable. We can summarize
our results concerning the embedding of distributions by saying:
\begin{enumerate}[leftmargin=*,label=(\roman*),align=left ]
\item The embedding of distributions by using a mollifier forces us to
take only one principal AG: $\B=\mathcal{Z}=\text{AG}(b)$.
\item If we are interested in using two different AG, $\B\ne\mathcal{Z}$,
or a non principal AG, we have to consider a particular set of indices,
e.g.\ the full one $\mathbb{I}^{\erm}$, where an intrinsic embedding
is possible. Of course, this is incompatible with particular properties
like $H(0)=0$.
\end{enumerate}

\section{\label{sec:LinHomODE}Solving linear homogeneous ODE with generalized
coefficients}

Studying Colombeau theory, one senses a sort of delusion by seeing
that these algebras, invented to find solutions of differential equations
which are not solvable in $\D'$, are not able to find solutions of
ODE of the simplest type. One way to bypass this problem is to assume
ad hoc growing conditions of logarithmic type, i.e.\ to adapt the
differential problem to the constraints of the theory (see e.g.\ \cite{Lig97,Lig98}
for linear ODE). Another solution is to guess that this deficiency
is due to the chosen polynomial growing condition, and that a generalization
could be possible. This is one of the basic motivations to generalize
Colombeau theory by defining notions like asymptotic scales, $(\mathcal{C},\mathcal{E},\mathcal{P})$
algebras, exponent weights or AG. Here, the point of view is more
similar to that used in algebra: given an equation we have to find
the best space where it has a solution, i.e.\ we adapt the theory
to the equation.

We start this section by defining the module of Colombeau generalized
numbers where we will take the coefficients of our linear ODE.
\begin{defn}
\label{def:Rtil}Let $\B$, $\mathcal{Z}$ be AG on a set of indeces
$\mathbb{I}$ such that $\R_{M}(\B)\subseteq\R_{M}(\mZ)$, and let
$d\in\N_{>0}$, then
\begin{enumerate}[leftmargin=*,label=(\roman*),align=left ]
\item $\Omega_{M}(\B):=\left\{ (x_{\eps})\in\Omega^{I}\mid\exists b\in\B:\ x_{\eps}=O(b_{\eps})\right\} $;
\item $(x_{\eps})\sim_{\mathcal{Z}}(y_{\eps})$ iff $\forall z\in\mathcal{Z}_{>0}:\ x_{\eps}-y_{\eps}=O(z_{\eps}^{-1})$,
where $(x_{\eps})$, $(y_{\eps})\in\Omega_{M}(\B)$;
\item $\widetilde{\Omega}(\B,\mZ):=\Omega_{M}(\B)/\sim_{\mZ}$;
\item $ $$\Rtil^{d}(\B,\mathcal{Z}):=\R_{M}^{d}(\B)/\sim_{\mathcal{Z}}$.
\end{enumerate}
\end{defn}
Like in Thm. \ref{thm:equivCondForIdeal}, we have that $\Rtil(\B,\mathcal{Z})$
is a ring. Moreover, $\Rtil(\B,\mathcal{Z})$ can be identified with
a subring of $\Rtil(\B',\mathcal{Z}')$ if
\begin{equation}
\R_{M}(\B)\subseteq\R_{M}(\B')\subseteq\R_{M}(\mathcal{Z}')\subseteq\R_{M}(\mathcal{Z}).\label{eq:inclusionsRtil}
\end{equation}
A sufficient condition for these inclusions is $\B\subseteq\B'\subseteq\mZ'\subseteq\mZ$.
The proof of Prop. 1.2.35 of \cite{GKOS} can be directly generalized
to every set of indices, so that if $\Omega$ is connected and $u\in\mathcal{G}(\B,\mathcal{Z},\Omega)$,
then $Du=0$ if and only if $u\in\Rtil(\B,\mathcal{Z})$.

As we mentioned above, if a differential equation $\dot{x}(t)=F(t,x(t))$
is well-defined in $\mathcal{G}(\B,\mathcal{Z},\Omega)^{n}$, i.e.\ if
$\Omega\subseteq\R^{1+n}$ and $F\in\mathcal{G}(\B,\mathcal{Z},\Omega)^{n}$,
then we will have to deal with moderate solutions bounded by terms
of the form $e^{b}:=\left(e^{b_{\eps}}\right)$, for some $b\in\B$.
It is therefore natural to set the following
\begin{defn}
\label{def:expB}Let $\B$ be an AG, then
\[
e^{\B}:=\left\{ e^{H\cdot b}\mid H\in\R_{>0},b\in\B\right\} 
\]
is called the \emph{exponential of $\B$}.
\end{defn}
\noindent The problem with $e^{\B}$ is that it is never a principal
AG since it always contains (bounds of) $e^{b^{m}}$, whereas a single
generator gives terms of the form $e^{mb}$.
\begin{lem*}
\label{lm:PropertiesOfexpB}Let $\B$ be an AG, then:
\begin{enumerate}[leftmargin=*,label=(\roman*),align=left ]
\item \label{enu:eBAG}$e^{\B}$ is a positive AG;
\item $\R_{M}(\B)\subseteq\R_{M}\left(e^{\B}\right)$;
\item $\R_{M}(e^{\B})=\R_{M}(e^{\R_{M}(\B)})$;
\item if $\B=\text{\emph{AG}}(b)$ then $\R_{M}\left(e^{\B}\right)=\R_{M}(\{e^{b^{k}}\mid k\in\N\})=\bigcup_{k\in\N}\R_{M}(\text{\emph{AG}}(e^{b^{k}}))\subseteq\R_{M}\left(\text{\emph{AG}}\left(e^{e^{b}}\right)\right)$;
\item \label{enu:expBNotPrincipal}$e^{\B}$ is not a principal AG.
\end{enumerate}
\end{lem*}
\begin{proof}
We only prove \ref{enu:expBNotPrincipal} since the other properties
follow almost directly from the definitions. Assume that $e^{\B}=\text{AG}\left(e^{H\cdot b}\right)$,
where $H\in\R_{>0}$ and $b\in\B$. Then $b^{2}=O(c)$, for some $c\in\B_{>0}$.
Therefore, for $\eps\in I$ sufficiently small we have
\begin{equation}
e^{b_{\eps}^{2}}\le e^{K\cdot c_{\eps}}\label{eq:b^2AndKc}
\end{equation}
for some $K\in\R_{>0}$. But $e^{Kc}\in e^{\B}=\text{AG}\left(e^{Hb}\right)$
so $e^{Kc}=\left(e^{Hb}\right)^{m}=e^{mHb}$ for some $m\in\N$. From
this and \eqref{eq:b^2AndKc} we get $b_{\eps}^{2}\le mHb_{\eps}$
for $\eps$ small. This implies that $b$ is bounded, so $e^{Hb}$
is also bounded and it cannot generate $e^{i}$, where $i\in\B_{>0}$
is infinite.
\end{proof}
We want to consider linear ODE whose coefficients are, in some sense,
``bounded by $\B$'', but whose solutions are in $\mathcal{G}(\B',\mZ',\R)$,
where $\R_{M}(\B')\supseteq\R_{M}\left(e^{\B}\right)$. We have to
clarify this point, also because it is desirable to have some kind
of preservation of old solutions: if $x$ is a solution already in
$\mathcal{G}(\B,\mZ,\R)$, e.g.\ because the coefficients have a
growth of logarithmic type, then $x$ must also be (in some sense)
the unique solution in the new space $\mathcal{G}(\B',\mZ',\R)$.
To have a relation between $\mathcal{G}(\B,\mZ,\R)$ and $\mathcal{G}(\B',\mZ',\R)$,
a condition like \eqref{eq:inclusionsRtil} is too strong because
if e.g.\ $\B=\mZ$ are of polynomial type, then \eqref{eq:inclusionsRtil}
implies that $\B'$ and $\mZ'$ cannot be of exponential type. On
the other hand, it is clear how to set the following
\begin{defn}
\label{def:boundedElement}Let $\B$, $\B'$, $\mZ'$ be AG such that
$\R_{M}(\B)\subseteq\R_{M}(\B')\subseteq\R_{M}(\mZ')$ and let $u\in\mathcal{G}(\B',\mZ',\Omega)$.
We say that $u$\emph{ is bounded by $\B$} if
\[
\exists(u_{\eps})\in\mathcal{E}_{M}(\B,\Omega):\ u=[u_{\eps}].
\]
We also set
\[
\mathcal{G}_{\B}(\B',\mZ',\Omega):=\left\{ u\in\mathcal{G}(\B',\mZ,\Omega)\mid u\text{ is bounded by }\B\right\} .
\]
Since element of $\Rtil(\B',\mZ')$ can be identified with constant
functions of $\mathcal{G}(\B',\mZ',\R)$, we have an analogous notion
for elements of the ring $\Rtil(\B',\mZ')$.
\end{defn}
Like in Lem.\ \ref{lem:negligibleAreModerate}, we can prove that
$u$ is bounded by $\B$ if and only if whenever we consider a representative
$u=[u_{\eps}]$, we have that $(u_{\eps})\in\mathcal{E}_{M}(\B,\Omega)$.
In the statement of the next result, we use the point value of a Colombeau
generalized function. We recall (see e.g.\ \cite{GiNi14} and references
therein) that this point value characterizes Colombeau generalized
functions:
\begin{defn}
\label{def:pointValue}Let $\B$, $\mathcal{Z}$ be AG such that $\R_{M}(\B)\subseteq\R_{M}(\mZ)$,
then
\begin{enumerate}[leftmargin=*,label=(\roman*),align=left ]
\item \foreignlanguage{british}{$[x_{\eps}]\in\otilc(\B,\mZ)$ iff $ $$[x_{\eps}]\in\widetilde{\Omega}(\B,\mZ)$
and $\exists K\Subset\Omega\,\forall^{\mathbb{I}}\eps:\ x_{\eps}\in K$.}
\item If $u=[u_{\eps}]\in\mathcal{G}(\B,\mZ,\Omega)$ and $ $$x\in\otilc(\B,\mZ)$,
then $u(x):=[u_{\eps}(x_{\eps})]$.
\end{enumerate}
\end{defn}
\begin{lem}
\label{lem:morphismOfAlgebras}Let $\B$, $\mZ$, $\B'$, $\mZ'$
be AG such that\foreignlanguage{british}{
\begin{equation}
\xymatrix{\R_{M}(\B)\ar@{^{(}->}[r]\ar@{^{(}->}[d] & \R_{M}(\mZ)\ar@{^{(}->}[d]\\
\R_{M}(\B')\ar@{^{(}->}[r] & \R_{M}(\mZ')
}
\label{eq:inclusionsDiagr}
\end{equation}
Then the following properties hold:}
\begin{enumerate}[leftmargin=*,label=(\roman*),align=left ]
\item \textup{\label{enu:boundedDiffSubAlg}$\mathcal{G}_{\B}(\B',\mZ',\Omega)$}\foreignlanguage{british}{
is a differential subalgebra of $\mathcal{G}(\B',\mZ',\Omega)$.}
\item \label{enu:boundedMorphism}The map
\[
\bar{(-)}:[u_{\eps}]\in\mathcal{G}_{\B}(\B',\mZ',\Omega)\mapsto(u_{\eps})+\mathcal{N}(\mZ,\Omega)\in\mathcal{G}(\B,\mZ,\Omega)
\]
is a surjective morphism of differential algebras.
\item \label{enu:boundedODE}Let $J:=(a,b)\subseteq\R$, $x\in\mathcal{G}_{\B}(\B',\mZ',(a,b))$,
$F\in\mathcal{G}_{\B}(\B',\mZ',(a,b)\times\Omega)$ be such that
\begin{equation}
\forall t\in\widetilde{J}_{c}(\B',\mZ'):\ t\text{ is bounded by }\B\Rightarrow\dot{x}(t)=F(t,x(t))\label{eq:ODE-In-B'Z'}
\end{equation}
holds in $\Rtil(\B',\mZ')$. Then
\begin{equation}
\dot{\bar{x}}(t)=\bar{F}(t,\bar{x}(t))\quad\forall t\in\widetilde{J}_{c}(\B,\mZ)\label{eq:ODE-In-BZ}
\end{equation}
holds in $\Rtil(\B,\mZ)$.
\end{enumerate}
\end{lem}
\begin{proof}
Property \ref{enu:boundedDiffSubAlg} follows by Lem.\ \ref{lem:moderateNegligibleRings}
and by the closure of $\mathcal{E}_{M}(\B,\Omega)$ with respect to
derivatives.

Property \ref{enu:boundedMorphism} follows by the $\eps$-pointwise
definitions of all the operations. The counter-image of $(u_{\eps})+\mathcal{N}(\mZ,\Omega)\in\mathcal{G}(\B,\mZ,\Omega)$
is $[u_{\eps}]$, which is bounded by $\B$ since $(u_{\eps})\in\mathcal{E}_{M}(\B,\Omega)$.

Assumption \eqref{eq:ODE-In-B'Z'} means that for each $t=[t_{\eps}]\in\widetilde{J}_{c}(\B',\mZ')$,
if $t$ is bounded by $\B$ then 
\begin{equation}
[x_{\eps}(t_{\eps})]\in\widetilde{\Omega}_{c}(\B',\mZ')\quad\text{and}\quad[\dot{x}_{\eps}(t_{\eps})]=[F_{\eps}(t_{\eps},x_{\eps}(t_{\eps}))].\label{eq:ODE-In-B'Z'-explicit}
\end{equation}
For simplicity, we use the symbol $\ldbrack u_{\eps}\rdbrack:=(u_{\eps})+\mathcal{N}(\mZ,\Omega)$
for the equivalence classes in $\mathcal{G}(\B,\mZ,\Omega)$ (and
hence also in $\Rtil(\B,\mZ)$). If $t=\ldbrack t_{\eps}\rdbrack\in\widetilde{J}_{c}(\B,\mZ)$,
then $[t_{\eps}]\in\Rtil(\B',\mZ')$ is bounded by $\B$ and \eqref{eq:ODE-In-B'Z'-explicit}
yields $[\dot{x}_{\eps}(t_{\eps})]=[F_{\eps}(t_{\eps},x_{\eps}(t_{\eps}))]$.
Both sides of this equality are bounded by $\B$, so that we can apply
the morphism $\bar{(-)}$ obtaining $\ldbrack\dot{x}_{\eps}(t_{\eps})\rdbrack=\ldbrack F_{\eps}(t_{\eps},x_{\eps}(t_{\eps}))\rdbrack$.
Moreover, $\ldbrack x_{\eps}(t_{\eps})\rdbrack\in\widetilde{\Omega}_{c}(\B,\mZ)$
because $[x_{\eps}]$ and $[t_{\eps}]$ are both bounded by $\B$.
We therefore have

\begin{align*}
\dot{\bar{x}}(t) & =\frac{d}{dt}\left(\ldbrack x_{\eps}\rdbrack\right)(\ldbrack t_{\eps}\rdbrack)=\ldbrack\dot{x}_{\eps}\rdbrack(\ldbrack t_{\eps}\rdbrack)=\\
 & =\ldbrack\dot{x}_{\eps}(t_{\eps})\rdbrack=\ldbrack F_{\eps}(t_{\eps},x_{\eps}(t_{\eps}))\rdbrack=\ldbrack F_{\eps}\rdbrack\left(\ldbrack t_{\eps},x_{\eps}(t_{\eps})\rdbrack\right)=\\
 & =\ldbrack F_{\eps}\rdbrack\left(\ldbrack t_{\eps}\rdbrack,\ldbrack x_{\eps}\rdbrack(\ldbrack t_{\eps}\rdbrack)\right)=\\
 & =\bar{F}\left(t,\bar{x}(t)\right)
\end{align*}

\end{proof}
Condition \ref{enu:boundedODE} states that any ODE framed in $\mathcal{G}(\B',\mZ',\Omega)$,
but restricted to elements which are bounded by $\B$, corresponds,
via the morphisms $\bar{(-)}$, to an ODE framed in $\mathcal{G}(\B,\mZ,\Omega)$.
This is our way to formalize that any bounded solution of an ODE of
bounded type in the ``bigger'' algebra $\mathcal{G}(\B',\mZ',\Omega)$
is also a solution in the ``smaller'' algebra $\mathcal{G}(\B,\mZ,\Omega)$.
The use of this order relation between algebras is formally introduced
in the following
\begin{defn}
\label{def:orderBetweenAlgebras}Let $\B$, $\mZ$, $\B'$, $\mZ'$
be AG, then we write $\mathcal{G}(\B,\mZ,\Omega)\preceq\mathcal{G}(\B',\mZ',\Omega)$,
and we say that $\mathcal{G}(\B,\mZ,\Omega)$\emph{ is smaller than
$\mathcal{G}(\B',\mZ',\Omega)$} if \eqref{eq:inclusionsDiagr} holds.
\end{defn}
The relation $\preceq$ is an order and, if $\mZ=\mZ'$, then $\mathcal{G}(\B,\mathcal{Z},\Omega)\preceq\mathcal{G}(\B',\mathcal{Z},\Omega)$
if and only if $\mathcal{G}(\B,\mathcal{Z},\Omega)\subseteq\mathcal{G}(\B',\mathcal{Z},\Omega)$
if and only if $\R_{M}(\B)\subseteq\R_{M}(\B')$. In this case, the
morphism $\bar{(-)}$ of Lem. \ref{lem:morphismOfAlgebras} is also
injective, and we have $\mathcal{G}_{\B}(\B',\mathcal{Z},\Omega)\simeq\mathcal{G}(\B,\mathcal{Z},\Omega)$.
We will use later the order relation $\preceq$.

In the following result, the main assumption is the inclusion $\R_{M}\left(e^{\B}\right)\subseteq\R_{M}(\B')$;
in it we can therefore set $\B'=e^{\B}$ or $\B'=\text{AG}\left(e^{e^{b}}\right)$
if we are interested to a principal AG.
\begin{thm}
\label{thm:linearConstantCoeffODE} Let $\B$, $\B'$, $\mZ'$ be
AG such that
\begin{equation}
\R_{M}\left(e^{\B}\right)\subseteq\R_{M}(\B')\subseteq\R_{M}(\mathcal{Z}').\label{eq:inclusionsAG}
\end{equation}
Let $t_{0}\in\R$, $c\in\Rtil(\B',\mZ')^{d}$ and $A\in\mathcal{M}_{d}(\Rtil(\B',\mZ'))$
be a $d\times d$ matrix with entries in the ring $\Rtil(\B',\mZ')$.
Assume that both $c$ and $A$ are bounded by $\B$. Then the problem

\begin{equation}
\left\{ \begin{array}{l}
x'(t)+A\cdot x(t)=0\\
{}x(t_{0})=c.
\end{array}\right.\label{eq:linConstODE}
\end{equation}

\noindent has a unique solution in $\mathcal{G}\left(\B',\mZ',\R\right)^{d}$.
\end{thm}
We split the proof of Theorem \ref{thm:linearConstantCoeffODE} in
two parts: existence and uniqueness.

To prove exitence we will use the following
\begin{lem}
\label{lem:crescitamatrice}Let $d\in\N_{>0}$, let $A=(a_{ij})_{i,j\leq d}\in\mathcal{M}_{d}(\mathbb{R})$
and let $M=\max_{i,j}|a_{ij}|$. Then for every entry $x_{ij}(t)$
of the matrix $e^{At}$, we have 
\[
|x_{ij}(t)|\leq M\cdot e^{dM|t|}.
\]
\end{lem}
\begin{proof}
For every $k\in\N$ let $A^{k}=:(a_{ijk})_{i,j\leq d}$ and let $M_{k}:=\max_{i,j}|a_{ijk}|$.
We claim that, for every $k\geq1,$ $M_{k}\leq d^{k-1}M^{k}$. Let
us prove this inequality by induction. If $k=1$ the conclusion is
trivial. Let us assume that the claim is true for $k$. Let us suppose
that $M_{k+1}=|a_{ij,k+1}|$. Then
\[
M_{k+1}=\left|\sum_{r,s=1}^{d}a_{irk}\cdot a_{sj1}\right|\leq\sum_{r,s=1}^{d}|a_{irk}\cdot a_{sj1}|\leq\sum_{r,s=1}^{d}M_{k}\cdot M\leq d^{k}M^{k+1}.
\]

\noindent The claim is proved. Therefore, since by definition $e^{At}=\sum_{k=0}^{\infty}\frac{A^{k}t^{k}}{k!}$,
we have
\[
|x_{ij}(t)|=\left|\sum_{k=0}^{\infty}\frac{a_{ijk}t^{k}}{k!}\right|\leq\sum_{k=0}^{\infty}\frac{|a_{ijk}||t|^{k}}{k!}\leq\sum_{k=0}^{\infty}\frac{d^{k}M^{k+1}|t|^{k}}{k!}=M\cdot e^{dM|t|},
\]
hence the thesis is proved.\end{proof}
\begin{lem}[Existence]
\label{lem:existence}Under the assumptions of Thm. \ref{thm:linearConstantCoeffODE},
the problem \eqref{eq:linConstODE} has a solution in $\mathcal{G}\left(\B',\mZ',\R\right)^{d}$.\end{lem}
\begin{proof}
Let $A=[A_{\eps}]$ and $c=[c_{\eps}]$. For every $\eps\in I$ let
$x_{\eps}\in\Coo(\R,\R^{d})$ be the unique solution of the problem
\begin{equation}
\left\{ \begin{array}{l}
x'(t)+A_{\eps}\cdot x(t)=0\\
x(t_{0})=c_{\eps}.
\end{array}\right.\label{eq:eqByRepresentatives}
\end{equation}
We claim that $[x_{\eps}]$ is a solution of \eqref{eq:linConstODE}
in $\mathcal{G}\left(\B',\mZ',\R\right)^{d}$. Since $\R_{M}\left(e^{\B}\right)\subseteq\R_{M}(\B')$,
in order to prove that $[x_{\eps}]$ is a solution in $\mathcal{G}\left(\B',\mZ',\R\right)^{d}$
it is sufficient to show that $(x_{\eps})\in\mathcal{E}_{M}\left(e^{\B},\R\right)^{d}$.
In fact, since also $x_{\eps}'$ verify the same equation \eqref{eq:eqByRepresentatives},
with initial $\B$-bounded value $x'_{\eps}(t_{0})=-A_{\eps}\cdot c_{\eps}$,
we can proceed by proving moderateness of $(x_{\eps})$ only. Without
loss of generality, we suppose $t_{0}=0$. For every $\eps\in I$
and $t\in\R$, we have
\[
x_{\eps}(t)=e^{-A_{\eps}t}c_{\eps}.
\]
For every $\eps$, set $A_{\eps}=:(a_{ij}{}_{\eps})_{i,j\leq d}$
for the entries of the matrix $A_{\eps}\in\mathcal{M}_{d}(\R)$, $M_{\eps}:=\max_{i,j}a_{ij}{}_{\eps}$,
$x_{\eps}(t)=:(x_{i\eps}(t))_{i\leq d}$, and $c_{\eps}=:(c_{i}{}_{\eps})_{i\leq d}$
for the components of the vectors $x_{\eps}(t)$, $c_{\eps}\in\R^{d}$.
By lemma \ref{lem:crescitamatrice} we deduce that 
\[
|x_{i\eps}(t)|\leq\sum_{j=1}^{d}M_{\eps}\cdot e^{dM_{\eps}|t|}c_{j}{}_{\eps}.
\]
Since both $\left[M_{\eps}\right]$ and $\left[c_{\eps}\right]$ are
bounded by $\B$, we have that $(\sum_{j=1}^{d}M_{\eps}\cdot e^{dM_{\eps}|t|}c_{j}{}_{\eps})\in\mathcal{E}_{M}\left(e^{\B},\R\right)$,
therefore $(x_{\eps})\in\mathcal{E}_{M}\left(e^{\B},\R\right)^{d}$. \end{proof}
\begin{lem}[Uniqueness]
\label{lem:uniqueness}Under the assumptions of Thm. \ref{thm:linearConstantCoeffODE},
if $x\in\mathcal{G}\left(\B',\mZ',\R\right)^{d}$ is such that

\begin{equation}
\left\{ \begin{array}{l}
x'(t)+A\cdot x(t)=0\\
{}x(t_{0})=0
\end{array}\right.\label{eq:uniqueness}
\end{equation}

\noindent in $\mathcal{G}\left(\B',\mZ',\R\right)$, then $x=0$.\end{lem}
\begin{proof}
Without loss of generality we can suppose that $t_{0}=0$. The generalized
function $x\in\mathcal{G}\left(\B',\mZ',\R\right)^{d}$ is a solution
of \eqref{eq:uniqueness} so there exist $(n_{\eps}),(v_{\eps})\in\mathcal{N}(\mZ',\R)$
such that, for every $\eps\in I$,

\begin{equation}
\left\{ \begin{array}{l}
x_{\eps}'(t)+A_{\eps}\cdot x_{\eps}(t)=n_{\eps}\\
x_{\eps}(0)=v_{\eps}.
\end{array}\right.\label{eq:0r}
\end{equation}
The unique solution of \eqref{eq:0r} is $x_{\eps}(t)=e^{-tA_{\eps}}v_{\eps}+\int_{0}^{t}e^{(s-t)A_{\eps}}\cdot n_{\eps}\diff{s}$
. So
\begin{align*}
|x_{\eps}(t)| & =\left|e^{-tA_{\eps}}v_{\eps}+\int_{0}^{t}e^{(s-t)A_{\eps}}\cdot n_{\eps}\diff{s}\right|\leq\\
 & \le e^{|t||A_{\eps}|}|v_{\eps}|+|t|e^{|t||A_{\eps}|}|n_{\eps}|,
\end{align*}
where we used the integral mean value theorem. If $K\Subset\R$, then
\[
\sup_{t\in K}\left|x_{\eps}(t)\right|\le e^{R|A_{\eps}|}|v_{\eps}|+Re^{R|A_{\eps}|}|n_{\eps}|,
\]
where $R:=\sup_{k\in K}|k|$. We have $\left(e^{R|A_{\eps}|}|v_{\eps}|+Re^{R|A_{\eps}|}|n_{\eps}|\right)\in\mathcal{N}(\mZ',\R)$
since $(v_{\eps}),(n_{\eps})\in\mathcal{N}(\mZ',\R)$ and $\left(e^{R|A_{\eps}|}\right),\left(Re^{R|A_{\eps}|}\right)\in\R_{M}(e^{\B})\subseteq\R_{M}(\B')$
because $A=[A_{\eps}]$ is bounded by $\B$.
\end{proof}
The results of Lemma \ref{lem:existence} and by Lemma \ref{lem:uniqueness}
provide a proof of Theorem \ref{thm:linearConstantCoeffODE}.
\begin{example}
Let $\mathbb{I}=(I,\le,\mathcal{I})$ be a set of indices, and $\rho:I\ra(0,1]$
be a map such that $\lim_{\mathbb{I}}\rho=0$. Let $\B^{\srm}$ be
the usual polynomial AG of the special Colombeau algebra, so that
$\B^{\srm}\circ\rho$ is an AG on $\mathbb{I}$ by Thm.\ \ref{thm:AGfromInfinitesRho}.
As we showed in Ex.\ \ref{exa:fullAlgebra}, this framework generalizes
the special, the full and the NSA based cases. The following problem:

\[
\left\{ \begin{array}{l}
x'(t)+\left[\frac{1}{\eps}\right]\cdot x(t)=0\\
{}x(0)=1
\end{array}\right.
\]

\end{example}
\noindent has not solution in $\mathcal{G}(\B^{\srm}\circ\rho,\Omega)$,
but it has a unique solution in $\mathcal{G}\left(e^{\B^{\srm}\circ\rho},\R\right)=\mathcal{G}\left(e^{\B^{\srm}\circ\rho},e^{\B^{\srm}\circ\rho},\R\right)$
and in $\mathcal{G}\left(\text{AG}\left(e^{e^{\frac{1}{\eps}}}\right)\circ\rho,\R\right)=\mathcal{G}\left(\text{AG}\left(e^{e^{\frac{1}{\eps}}}\right)\circ\rho,\text{AG}\left(e^{e^{\frac{1}{\eps}}}\right)\circ\rho,\R\right)$,
namely
\[
[x_{\eps}(t)]=\left[e^{-\frac{1}{\eps}t}\right],
\]
where the equivalence class has to be meant differently in the two
algebras. Let us note explicitly that we have applied Thm. \ref{thm:linearConstantCoeffODE}
with $\B=\B^{\srm}\circ\rho$ and $\B'=\mZ'=e^{\B^{\srm}\circ\rho}$
for the former algebra and $\B'=\mZ'=\text{AG}\left(e^{e^{\frac{1}{\eps}}}\right)$
for the latter. Moreover, this problem has also a unique solution
in the algebra $\mathcal{G}\left(\text{AG}\left(e^{\frac{1}{\eps}}\right)\circ\rho,\R\right).$
This shows one particular feature of our construction: if we want
to have an algebra in which we can uniquely solve all the ODEs whose
coefficients are bounded by a given AG $\B$ then (as we will show
in Thm. \ref{thm:expB-smallest}) the minimal possible choice is $\mathcal{G}\left(e^{\B},\R\right)$,
whilst if we are interested only in a finite number of linear ODE
with coefficients bounded by $\B$, then it is possible to find a
solution to these ODE in the algebra $\mathcal{G}\left(\text{AG}\left(e^{b}\right),\R\right),$
where $b\in\B_{>0}$ is any element such that $|c|=O(b)$ for every
coefficient $c$ that appears in the finite set of ODE.

We note that Theorem \ref{thm:linearConstantCoeffODE} can be reformulated
in the following way: 
\begin{thm}
\label{thm:expB-smallest}Let $\B$ be an AG. Then $\mathcal{G}\left(e^{\B},\R\right)$
is the smallest Colombeau algebra (with respect to the order relation
$\preceq$ of Def. \ref{def:orderBetweenAlgebras}) in which every
linear homogeneous ODE with coefficients in $\R_{M}(\B)$ can be solved.\end{thm}
\begin{proof}
By Theorem \ref{thm:linearConstantCoeffODE} we know that every linear
homogeneous ODE with coefficients in $\R_{M}(\B)$ can be solved in
$\mathcal{G}\left(e^{\B},e^{\B},\R\right)$. Now let $\mathcal{B}'$
be an AG such that every linear homogeneous ODE with coefficients
bounded by $\B$ can be solved in $\mathcal{G}\left(\mathcal{B}',\mathcal{B}',\R\right)$.
In particular, for every $b\in\B$ we can solve the problem

\[
\left\{ \begin{array}{l}
x'(t)+b\cdot x(t)=0\\
{}x(0)=1.
\end{array}\right.
\]

\noindent As we showed in Lemma \ref{lem:existence}, the solution
of this problem is $\left[e^{b_{\eps}t}\right]$. This means that
$(e^{b_{\eps}t})\in\mathcal{E}(\B',\R)$ for every $b\in\B$. In particular
this entails that $(e^{H\cdot b_{\eps}})\in\R_{M}(\B')$ for every
$b\in\B$ and $H\in\R_{>0}$, so $e^{\B}\subseteq\R_{M}(\B')$ and
$\R_{M}(e^{\B})\subseteq\R_{M}(\B')$. Therefore, condition \eqref{eq:inclusionsDiagr}
holds for $\mZ=e^{\B}$ and $\mZ'=\B'$ so $\mathcal{G}\left(e^{\B},e^{\B},\R\right)\preceq\mathcal{G}\left(\B',\B',\R\right)$.\end{proof}

\end{document}